\documentclass[11pt,a4paper,reqno]{amsart}
\usepackage{amsmath, amssymb, amsthm, graphicx, hyperref, cleveref, mathrsfs, euscript, upgreek, stmaryrd, leftidx, mathtools, pb-diagram, tikz-cd, pigpen, bbm, multirow, wrapfig, epsfig, epsf, epic, pdfsync, tikz, color, xcolor, xypic, indentfirst, textcomp, tcolorbox, fancyhdr, blindtext, appendix, relsize, multicol, cutwin}

\usepackage[margin=2cm]{geometry}
\usepackage{enumitem}
\usepackage[export]{adjustbox}

\usepackage{libertinus}
\usepackage[T1]{fontenc}

\newtheorem*{namedthm}{\namedthmname}    	%Referencing unnumbered theorems
\newcounter{namedthm}

\makeatletter
\newenvironment{named}[1]
{\def\namedthmname{#1}%
	\refstepcounter{namedthm}%
	\namedthm\def\@currentlabel{#1}}
{\endnamedthm}
\makeatother

\theoremstyle{plain}
\newtheorem{thm}{Theorem}[section]
\newtheorem{theorem}[thm]{Theorem}
\newtheorem{lemma}[thm]{Lemma}
\newtheorem{prop}[thm]{Proposition}
\newtheorem{corollary}[thm]{Corollary}
\theoremstyle{definition}
\newtheorem{definition}[thm]{Definition}

\newtheorem{example}[thm]{Example}

\newtheorem{remark}[thm]{Remark}
\numberwithin{equation}{section}

\newtheorem*{theorem*}{Theorem}

\NewDocumentCommand{\myparen}{m}{\textup{(#1)}}
\newcommand{\set}{\EuScript{S}\mathsf{et}}
\newcommand{\sset}{\mathsf{s}\EuScript{S}\mathsf{et}}

\newcommand*{\pb}{\mbox{\LARGE{$\mathrlap{\cdot}\lrcorner$}}}

\providecommand{\bysame}{\leavevmode\hbox to3em{\hrulefill}\thinspace}
\providecommand{\MR}{\relax\ifhmode\unskip\space\fi MR }
\providecommand{\MRhref}[2]{%
	\href{http://www.ams.org/mathscinet-getitem?mr=#1}{#2}
}

\begin{document}
	
\title{Reshaping limit diagrams and cofinality in higher category theory}
\author{Peng Du}

\begin{abstract}
	In this article, we present some results on (co)limits of diagrams in $\infty$-categories, as well as those in $(n, 1)$-categories. In particular, we deduce a way to reshape colimit diagrams into simplicial ones, and some characterisations of $n$-cofinality for functors between $\infty$-categories. Some basics on $n$-siftedness are also treated.
\end{abstract}

\maketitle
\makeatletter
\patchcmd{\@tocline}
{\hfil}
{\leaders\hbox{\,.\,}\hfil}{}{}
\makeatother
{\large {\tableofcontents}}

\section{Introduction}

In this article, we discuss some results in higher category theory. The first main theme is to reshape colimit diagrams into simplicial ones. Another one is $n$-cofinality, which will give a simplification of the shape if the diagrams are actually in $(n, 1)$-categories. As an application of $n$-cofinality, we also present some fundamental results on $n$-siftedness, extending the usual siftedness in ordinary category theory and the very useful notion of siftedness for $\infty$-categories.

The motivation and a prototypical example is the result \cite[Lemma 18.9.1]{Hir03}, which computes homotopy colimits of diagrams of the shape of the opposite to the category of simplices of another simplicial set, and reshape such diagrams to simplicial ones. On the other hand,  a (co)limit is a (co)equaliser of (co)products in any ordinary category. We seek for a common generalisation of these facts.

By means of computing limits by decomposing diagrams, we can give a quick and non-technical proof of such a result for (co)limits of diagrams of the shape $\mathbf{\Delta}_{/C}$ for a simplicial set $C$ in the $\infty$-categorical setting (\Cref{lim-catsimp}). Based on this and taking advantage of the well-developped $\infty$-categorical techniques, we can indeed formulate a common generalisation of these facts for $\infty$-categories, and find a shortcut to confirm it, which we summarise as our first main result, in the following form.
\begin{named}{\textsc{Main Theorem 1}}
	Let $C$ be a simplicial set and let $\EuScript{E}$ be an $\infty$-category. Let $F\colon C\to\EuScript{E}$ be a diagram. We have a natural equivalence
	\[\operatorname{colim}_{[r]\in\mathbf{\Delta}^{\rm op}}\coprod_{\sigma\in C_r}F(\sigma(0))\simeq\operatorname{colim}(C\xrightarrow{F}\EuScript{E})\]
	whenever all terms make sense, in which case, we have a colimit diagram
	\newcommand{\stackspacesemi}{4}
	\newcommand{\stackk}[2][.8cm]{\;\tikz[baseline, yshift=.65ex]%
		{\foreach \k [evaluate=\k as \r using (.5*#2+.5-\k)*\stackspacesemi] in {1,...,#2}{%
				\k{\draw[->](0,\r pt)--(#1,\r pt);}
		}}\;}
	\[
	\begin{tikzcd}
		\cdots\cdots\stackk{4}\coprod_{\sigma\in C_2}F(\sigma(0))\stackk{3}\coprod_{\sigma\in C_1}F(\sigma(0)) \stackk{2}\coprod_{\sigma\in C_0}F(\sigma(0)) \arrow[r]   & \operatorname{colim}(C\xrightarrow{F}\EuScript{E}).
	\end{tikzcd}\]
	
	If $\EuScript{E}$ is an $(n, 1)$-category, then the same holds if we truncate the diagram at level $n$:
	\[\operatorname{colim}_{[r]\in(\mathbf{\Delta}^{\leqslant n})^{\rm op}}\coprod_{\sigma\in C_r}F(\sigma(0))\simeq\operatorname{colim}(C\xrightarrow{F}\EuScript{E}).\]
\end{named}
It should be a folklore result and might already exists somewhere; I cannot find in the literature only due to my ignorance. 

To demonstrate, we only need a seemingly very weak special application of \Cref{lim-catsimp}, namely, for a simplicial set $C$, we can identify $\operatorname{colim}_{[n]\in\mathbf{\Delta}^{\rm op}}C_n$ with its homotopy type (\Cref{htyp-sset}).

The first part of the above theorem is given in \Cref{colimits-simplicial}. The last part is a consequence of \Cref{simp-ncofinal} together with our second main result (on $n$-cofinality), \Cref{$n$-cofinal-lim-pres}.
\begin{named}{\textsc{Main Theorem 2}}
	A functor $p\colon\EuScript{C}\to\EuScript{D}$ between small $\infty$-categories is right $n$-cofinal if and only if $p$ respects colimits in $(n, 1)$-categories: for every $(n, 1)$-category $\EuScript{E}$ and every functor $F\colon\EuScript{D}\to\EuScript{E}$, we have
	\[\operatorname{colim}p^*F\xrightarrow{\simeq}\operatorname{colim}F.\]
\end{named}
In fact, our \Cref{$n$-cofinal-lim-pres} characterises $n$-cofinality, another main theme of this article, in a few equivalent ways and generalises a form of Quillen's Theorem A (see e.g. \cite[Theorem 6.4.5]{CiHi} and \cite[Theorem 2.19]{COJ}). We then commence to a relatively thorough study in \S 6 on $n$-cofinality, largely base on this. Those familiar with ($\infty$-)cofinality can easily recognise its similarity with the corresponding $\infty$-categorical version as in \cite{HTT, kerodon} of Lurie. Such characterisation via respecting of (co)limits is useful --- for instance, one can easily deduce that a constant diagram in an $(n, 1)$-category indexed by an $n$-connective simplicial set has that constant object as its colimit (see \Cref{cofintopt-cst}); it is also easy to apply in practice, as we may see in later part of \S 6.

\vspace{3mm}

The \emph{essence} turns out to be the following: if an $\infty$-category is an $(n, 1)$-category, while it is correct to use notions $\infty$-categorically, we could truncate at level $n$ without losing anything, or should loosen condition to level $n$ (rather than up to level $\infty$) to have the correct notions for $(n, 1)$-categories. This also justifies our notion of being left/right $n$-cofinal in \Cref{$n$-cofinality}. See \Cref{lim-trunc-des} for further comments in this aspect.

As an illustration, a (non-empty small) $\infty$-category $\EuScript{C}$ is sifted if and only if, for every pair of objects $a, b\in\EuScript{C}$, the underlying simplicial set $\EuScript{C}_{a/}\times_{\EuScript{C}}\EuScript{C}_{b/}$ is weakly contractible (in the Kan-Quillen model structure on $\sset$); for an ordinary category (i.e. a $(1, 1)$-category) $\EuScript{C}$ to be $1$-sifted however, we only require $\EuScript{C}_{a/}\times_{\EuScript{C}}\EuScript{C}_{b/}$ to be (non-empty and) connected. One may infer that the correct notion of ($n$-)siftedness suits for the study of $(n, 1)$-categories should require $\EuScript{C}_{a/}\times_{\EuScript{C}}\EuScript{C}_{b/}$ to be (non-empty and) $n$-connective, for all objects $a, b\in\EuScript{C}$; this will in particular diminish some confusion like conflict of terminology as noted in \cite[Warning 5.5.8.2]{HTT}. We will discuss $n$-siftedness in \S 7. Our main result regarding $n$-siftedness is as follows (see \Cref{nsift-colim-pres}).
\begin{named}{\textsc{Main Theorem 3}}
	Let $\EuScript{C, D}$ be $(n, 1)$-categories with $\EuScript{C}$ cocomplete, let $F\colon\EuScript{C}\to\EuScript{D}$ be a functor. Then $F$ preserves $n$-sifted colimits if and only if $F$ preserves filtered colimits and $(\mathbf{\Delta}^{\leqslant n})^{\rm op}$-shaped colimits.
\end{named}

\vspace{3mm}

Let me emphasise again that these are largely folklore results and should be familiar by experts (but they do not bother to write down). A main reason for writing up this article is, on one hand, the author feels it hard to find a reference where such results are properly recorded; on the other, the author hopes that the results presented here will be useful to others, especially for novices in $\infty$-categories who are familiar with ordinary categories, hence would probably expect such folklore results like me.

\vspace{3mm}

Throughout, we use the language of $\infty$-category theory as established in \cite{HTT, kerodon}. In particular, we use Joyal's theory of quasi-categories as a model for $(\infty, 1)$-categories. Since notions and results of Lurie are used almost everywhere, we do not always spell out explicitly when we use them but will be easily identified by readers familiar with \cite{HTT}. The notation ${\rm Fun}^{\rm bla}(-, -)$ is to indicate the full subcategory of functors in ${\rm Fun}(-, -)$ that preserve certain operation or structure $\emph{bla}$ (e.g. ${\rm Fun}^{\operatorname{lim}}(\EuScript{C}, \EuScript{E})\subset{\rm Fun}(\EuScript{C}, \EuScript{E})$ consists of limit-preserving functors). We also make the convention that $\mathbf{\Delta}^{\leqslant n}=\varnothing$ if $n<0$, and  $\mathbf{\Delta}^{\leqslant n}=\mathbf{\Delta}$ if $n=\infty$.

\vspace{3mm}

{\noindent\bf Acknowledgements.} I thank Dustin Clausen for inspiring discussions on a few things in this article during a stay at Bristol, and especially for pointing out the work \cite{COJ}. I also thank Denis-Charles Cisinski for helping with the argument of \Cref{htyp-sset} here, and Daniel Gratzer for pointing out the reference \cite{HP22}. The author gratefully acknowledges the support of the EPSRC standard grant EP/T012625/1.

\section{Cofinality}

In this section, we recall some facts about cofinal functors between $\infty$-categories for later use.

Let $\EuScript{C, D}$ be small $\infty$-categories, we say that a functor $p\colon\EuScript{C}\to\EuScript{D}$ is \emph{right cofinal}, if for every object $d\in\EuScript{D}$, the simplicial set $\EuScript{C}_{d/}:=\EuScript{C}\times_{\EuScript{D}}\EuScript{D}_{d/}$ is weakly contractible (in the Kan-Quillen model structure on $\sset$); this notation could be misleading if $d$ is a common object of $\EuScript{C}$ and $\EuScript{D}$, e.g. if $p\colon\EuScript{C}\to\EuScript{D}$ is the embedding of a wide subcategory --- a better way might be to include $p$ in the notation. A functor is \emph{left cofinal} if its opposite is right cofinal.

Being left/right cofinal is invariant under categorical equivalences. A map is left anodyne if and only if it is a left cofinal monomorphism, a map is right anodyne if and only if it is a right cofinal monomorphism. A left cofinal map factors as a composition of a left anodyne map followed by a trivial Kan fibration, a right cofinal map factors as a composition of a right anodyne map followed by a trivial Kan fibration.

\begin{example}
	 \emph{The inclusion functor $\mathbf{\Delta}_{\rm s}\hookrightarrow\mathbf{\Delta}$ is left cofinal.} See \cite[Lemma 6.5.3.7]{HTT}.
	
	In fact, for $[m]\in\mathbf{\Delta}$, we have the following commutative diagram, with a pullback square on the right:
\begin{equation}\label{cof-simpidx}
	\begin{tikzcd}
		(\mathbf{\Delta}_{\rm s}^{\leqslant n})_{/[m]} \ar[r]    \ar[d, hook]      & \mathbf{\Delta}_{/[m]}\times_{\mathbf{\Delta}}\mathbf{\Delta}_{\rm s} \ar[r] \ar[d, hook]  \arrow[dr, phantom, "\pb", very near start]   & \mathbf{\Delta}_{\rm s} \arrow[d, hook]\\
		(\mathbf{\Delta}^{\leqslant n})_{/[m]}  \ar[r]    & \mathbf{\Delta}_{/[m]} \ar[r]      & \mathbf{\Delta},
	\end{tikzcd}
\end{equation}
	in which, \emph{the leftmost vertical arrow is right cofinal} (see the proof of \cite[Lemma 6.5.3.8]{HTT}) and \emph{the other two vertical inclusions are left cofinal} (for the middle vertical arrow, we use that the arrow $\mathbf{\Delta}_{/[m]}\to\mathbf{\Delta}$ is a cartesian fibration and that left cofinal maps are stable under pullback along cartesian fibrations --- cartesian fibrations are smooth). Here we write $(\mathbf{\Delta}_{\rm s}^{\leqslant n})_{/[m]}:=(\mathbf{\Delta}_{\rm s})_{/[m]}\times_{\mathbf{\Delta}_{\rm s}}\mathbf{\Delta}_{\rm s}^{\leqslant n}$ (so that the square on the left is not a pullback in general). Note that, if $m\leqslant n$, then $(\mathbf{\Delta}_{\rm s}^{\leqslant n})_{/[m]}\cong(\mathbf{\Delta}_{\rm s}^{\leqslant m})_{/[m]}\cong(\mathbf{\Delta}_{\rm s})_{/[m]}$ in fact has a final object.
\end{example}
\begin{example}\label{ex-rcofmaps}
For any $n\geqslant 0$, the inclusion $\Delta^0\xrightarrow{n=d^{n-1}\circ\cdots\circ d^1\circ d^0}\Delta^n$ is right anodyne (\cite[Lemma 4.1.15]{CiHi}). If $0<r\leqslant n$, the inclusion $\Delta^0\xhookrightarrow{n}\Lambda^n_r$ is right anodyne (\cite[Lemma 4.4.3]{CiHi}). So all three inclusions $\Delta^0\xhookrightarrow{n}\Lambda^n_r\hookrightarrow\Delta^n$ are right anodyne. Any map $\Delta^0\xrightarrow{a}C$ in $\sset$ yields a right anodyne map $\Delta^0\xrightarrow{a}C_{/a}$ (\cite[Corollary 4.3.8]{CiHi}). All these maps are right cofinal.
\end{example}

\section{Functor-fibration correspondences}

In this section, we state Lurie's straightening-unstraightening equivalence (\cite[Theorem 3.2.0.1]{HTT}), also known as Lurie-Grothendieck correspondence. It is a nice way to treat the subtle problem of describing ${\rm CAT}_{\infty}$-valued functors by means of (co)cartesian fibrations. This is to fix notations and to state it in a form easy for us to apply.

We do not repeat definition of (co)cartesian fibrations here, which can be found in \cite[Chapter 2]{HTT}. We only emphasise that the class of (co)cartesian fibrations between simplicial sets is stable under pullback and compositions ({\cite[Proposition 2.4.2.3]{HTT}}).

For a small $\infty$-category $\EuScript{C}$, we denote by ${\rm CAT}_{\infty/\EuScript{C}}^{\rm cart}, {\rm CAT}_{\infty/\EuScript{C}}^{\rm cocart}, {\rm CAT}_{\infty/\EuScript{C}}^{\rm bicart}$ the (non-full) $\infty$-subcategories of ${\rm CAT}_{\infty/\EuScript{C}}$ with objects respectively being cartesian fibration, cocartesian fibration, bicartesian fibration, over $\EuScript{C}$, and with morphisms the functors (over $\EuScript{C}$) preserving cartesian arrows, cocartesian arrows, bicartesian arrows, respectively. We have ${\rm CAT}_{\infty/\EuScript{C}^{\rm op}}^{\rm cart}={\rm CAT}_{\infty/\EuScript{C}}^{\rm cocart}, {\rm CAT}_{\infty/\EuScript{C}^{\rm op}}^{\rm cocart}={\rm CAT}_{\infty/\EuScript{C}}^{\rm cart}$.
\begin{theorem}\label{unsteq}
	Let $\EuScript{C}$ be a small $\infty$-category. We have \emph{unstraightening} equivalences
	\begin{equation}
		\begin{aligned}
			&{\rm Un}=\int^{\EuScript{C}}\colon{\rm Fun}(\EuScript{C}^{\rm op}, {\rm CAT}_{\infty})\xrightarrow{\sim}{\rm CAT}_{\infty/\EuScript{C}}^{\rm cart}, \\
			&{\rm Un}=\int_{\EuScript{C}}\colon{\rm Fun}(\EuScript{C}, {\rm CAT}_{\infty})\xrightarrow{\sim}{\rm CAT}_{\infty/\EuScript{C}}^{\rm cocart}\simeq{\rm CAT}_{\infty/\EuScript{C}^{\rm op}}^{\rm cart}.
		\end{aligned}
	\end{equation}
	
	The inverse \emph{straightening} are given by taking fibre $\infty$-categories. They are related by a commutative diagram
	\begin{equation}
		\begin{tikzcd}[column sep=2em]
			{\rm CAT}_{\infty/\EuScript{C}}^{\rm cart} \arrow[d, "{\rm op}"']	 \ar[r, "{\rm Str}"]  	&	{\rm Fun}(\EuScript{C}^{\rm op}, {\rm CAT}_{\infty}) \ar[r, "\int^{\EuScript{C}}"] \ar[d, "{\rm op}\circ"']  & {\rm CAT}_{\infty/\EuScript{C}}^{\rm cart} \arrow[d, "{\rm op}"]\\
			{\rm CAT}_{\infty/\EuScript{C}^{\rm op}}^{\rm cocart} \ar[r, "{\rm Str}"]  	&	{\rm Fun}(\EuScript{C}^{\rm op}, {\rm CAT}_{\infty})  \ar[r, "\int_{\EuScript{C}^{\rm op}}"]     & {\rm CAT}_{\infty/\EuScript{C}^{\rm op}}^{\rm cocart}
		\end{tikzcd}
	\end{equation}
	of equivalences, where the middle vertical arrow is given by composing with the auto-equivalence ${\rm op}\colon{\rm CAT}_{\infty}\to{\rm CAT}_{\infty}, \EuScript{D}\mapsto\EuScript{D}^{\rm op}$. So we have
	\[
	\left(\int^{\EuScript{C}}F\to\EuScript{C}\right)^{\rm op}\simeq\left(\int_{\EuScript{C}^{\rm op}}{\rm op}\circ F\to\EuScript{C}^{\rm op}\right)\in{\rm CAT}_{\infty/\EuScript{C}^{\rm op}}^{\rm cocart}\]
	for any $F\in{\rm Fun}(\EuScript{C}^{\rm op}, {\rm CAT}_{\infty})$.
	
	The unstraightening equivalences are compatible with base change$\colon$for a functor $f\colon\EuScript{C}\to\EuScript{D}$ of $\infty$-categories, we have a commutative diagram
	\[
	\begin{tikzcd}[column sep=2em]
		{\rm Fun}(\EuScript{D}, {\rm CAT}_{\infty}) \ar[r, "{\rm Un}"] \ar[d, "f^*"']  & {\rm CAT}_{\infty/\EuScript{D}}^{\rm cocart} \arrow[d, "f^*"]\\
		{\rm Fun}(\EuScript{C}, {\rm CAT}_{\infty})  \ar[r, "{\rm Un}"']     & {\rm CAT}_{\infty/\EuScript{C}}^{\rm cocart},
	\end{tikzcd}\]
	where the right vertical $f^*$ sends a cocartesian fibration $\EuScript{E}\to\EuScript{D}$ to $\EuScript{E}\times_{\EuScript{D}}\EuScript{C}\to\EuScript{C}$.
\end{theorem}
For a functor $\varphi\colon\EuScript{C}\to{\rm CAT}_{\infty}$ corresponding to a cocartesian fibration $p\colon\EuScript{E}\to\EuScript{C}$ by unstraightening, we say that the cocartesian fibration $p$ is classified by the functor $\varphi\colon\EuScript{C}\to{\rm CAT}_{\infty}$.

We now give some examples that will be useful later.
\begin{example}
	Let $\EuScript{C, D}$ be $\infty$-categories, then the projection $\EuScript{C}\times\EuScript{D}\to\EuScript{C}$ is a bicartesian fibration, whose straightening are the constant ${\rm CAT}_{\infty}$-valued functors with value $\EuScript{D}$.
\end{example}
\begin{example}
	Let $\EuScript{C}$ be a small $\infty$-category, let $X\in\EuScript{P}(\EuScript{C})$. Then the projection $q\colon\EuScript{C}_{/X}\to\EuScript{C}$ is a right fibration classified by the functor $X\colon\EuScript{C}^{\rm op}\to\EuScript{S}$.
\end{example}
\begin{example}\label{catsimp-cart}
	Let $C$ be a simplicial set.
	\begin{enumerate}[label=(\arabic*)]
		\item The identity functor $\operatorname{id}\colon C\to C$
		is a bicartesian fibration, with all edges in $C$ being $\operatorname{id}$-bicartesian. It is classified by the constant functor
		\[
		\begin{aligned}
			\underline{*}\colon C&\to\EuScript{S},\\
			a&\mapsto *.
		\end{aligned}\]
		\item Let $\mathbf{\Delta}_{/C}$ be the category of simplices of $C$ (constructed relative to the inclusion functor $\mathbf{\Delta}\hookrightarrow\sset, [n]\mapsto\Delta^n$). The functor
		\[
		\begin{aligned}
			p\colon\mathbf{\Delta}_{/C}&\to\mathbf{\Delta},\\
			(\Delta^n\xrightarrow{\sigma}C)&\mapsto [n]
		\end{aligned}\]
		is a cartesian fibration, with all edges in $\mathbf{\Delta}_{/C}$ being $p$-cartesian. So $p^{\rm op}\colon(\mathbf{\Delta}_{/C})^{\rm op}\to\mathbf{\Delta}^{\rm op}$
		is a cocartesian fibration, with all edges in $(\mathbf{\Delta}_{/C})^{\rm op}$ $p^{\rm op}$-cocartesian. It is classified by the functor
		\[
		\begin{aligned}
			C\colon\mathbf{\Delta}^{\rm op}&\to\EuScript{S},\\
			[n]&\mapsto C_n.
		\end{aligned}\]
	\end{enumerate}
\end{example}
\begin{example}\label{ar-cart-fib}
Let $\EuScript{C}$ be an $\infty$-category, then the functor $d_0\colon\EuScript{C}^{\Delta^1}\to\EuScript{C}$ sending an arrow in $\EuScript{C}$ to its target is a cocartesian fibration (\cite[Corollary 2.4.7.12]{HTT} applied to the opposite of the identity functor of $\EuScript{C}$). An edge $\sigma$ in $\EuScript{C}^{\Delta^1}$ is $d_0$-cocartesian if and only if its source $d_1(\sigma)$ is an equivalence (as an edge in $\EuScript{C}$). Straightening $d_0\colon\EuScript{C}^{\Delta^1}\to\EuScript{C}$, we obtain a functor
\[\EuScript{C}\to\operatorname{CAT}_{\infty}, c\mapsto\EuScript{C}_{/c}.\]

If $\EuScript{C}$ admits fibered products, then the functor $d_0\colon\EuScript{C}^{\Delta^1}\to\EuScript{C}$ is also a cartesian fibration. An edge in $\EuScript{C}^{\Delta^1}$ is $d_0$-cartesian if and only if it is a cartesian square in $\EuScript{C}$ (\cite[Lemma 6.1.1.1]{HTT}). Straightening $d_0\colon\EuScript{C}^{\Delta^1}\to\EuScript{C}$, we obtain a functor
\[F_{\EuScript{C}}\colon\EuScript{C}^{\mathrm{op}}\to\operatorname{CAT}_{\infty}, c\mapsto\EuScript{C}_{/c}.\]
\end{example}

\section{Limits by decomposing diagrams}

In this section, we recall some results on limits and colimits in $\operatorname{Cat}_{\infty}$, with emphasis on the technique of decomposing diagrams (as in \cite[\S 2]{COJ}).
\begin{prop}\label{bas-crit-catcolim}
	Let $K$ be a simplicial set and let $\varphi\colon K^{\triangleright}\to\operatorname{Cat}_{\infty}$ be a diagram. Then the following are equivalent.
	\begin{enumerate}[label=\myparen{\arabic*}]
		\item $\varphi\colon K^{\triangleright}\to\operatorname{Cat}_{\infty}, i\mapsto\EuScript{C}(i)$ is a colimit diagram.
		\item For every $\EuScript{E}\in\operatorname{Cat}_{\infty}$, the functor $\operatorname{Map}_{\operatorname{Cat}_{\infty}}(\varphi, \EuScript{E})\colon(K^{\rm op})^{\triangleleft}\to\EuScript{S}$ is a limit diagram.
		\item For every $\EuScript{E}\in\operatorname{Cat}_{\infty}$, the functor $\operatorname{Fun}(\varphi, \EuScript{E})\colon(K^{\rm op})^{\triangleleft}\to\operatorname{Cat}_{\infty}$ is a limit diagram.
	\end{enumerate}
\end{prop}
\begin{proof}
	The equivalence of (1) and (2) is by definition, (3)$\Rightarrow$(2) since the functor $\operatorname{Cat}_{\infty}\to\EuScript{S}, \EuScript{C}\mapsto\EuScript{C}^{\simeq}$ preserves limits.
	
	(2)$\Rightarrow$(3): we need to show that, for every $\EuScript{F}\in\operatorname{Cat}_{\infty}$, the functor $\operatorname{Map}_{\operatorname{Cat}_{\infty}}(\EuScript{F}, \operatorname{Fun}(\varphi, \EuScript{E}))\colon(K^{\rm op})^{\triangleleft}\to\EuScript{S}$ is a limit diagram. But $\operatorname{Fun}(\EuScript{F}, \operatorname{Fun}(\varphi, \EuScript{E}))=\operatorname{Fun}(\varphi, \operatorname{Fun}(\EuScript{F}, \EuScript{E}))$, so the result follows by taking cores and (2).
\end{proof}

The following result describes limits and colimits in $\operatorname{Cat}_{\infty}$. See \cite[Corollaries 3.3.3.2 and 3.3.4.3]{HTT} for proofs.
\begin{theorem}\label{lim-catinf-op}
	Let $K$ be a simplicial set, let $\varphi\colon K^{\mathrm{op}}\to\operatorname{Cat}_{\infty}$ be a functor classifying a cartesian fibration $p\colon X=\int^K\varphi\to K$. Then $\operatorname{lim}\varphi$ is a model of cartesian sections of $p$, and $\operatorname{colim}\varphi$ is a localisation of $X$ by $p$-cartesian edges.
	
	If $K$ is an $\infty$-category, then
	\[\operatorname{lim}\varphi\simeq\Gamma_{\rm cart}(p)\simeq\operatorname{Fun}_{/K}^{\rm cart}(K, X)\]
	consists of functors that send every edge in $K$ to a $p$-cartesian edge in $X$.
\end{theorem}
Here is a dual version.
\begin{theorem}\label{lim-catinf}
	Let $K$ be a simplicial set, let $\varphi\colon K\to\operatorname{Cat}_{\infty}$ be a functor classifying a cocartesian fibration $p\colon X=\int_K\varphi\to K$. Then $\operatorname{lim}\varphi$ is a model of cocartesian sections of $p$, and $\operatorname{colim}\varphi$ is a localisation of $X$ by $p$-cocartesian edges.
	
	If $K$ is an $\infty$-category, then
	\[\operatorname{lim}\varphi\simeq\Gamma_{\rm cocart}(p)\simeq\operatorname{Fun}_{/K}^{\rm cocart}(K, X)\]
	consists of functors that send every edge in $K$ to a $p$-cocartesian edge in $X$.
\end{theorem}

We can now make it precise the idea that limits and colimits over a diagram, which breaks up into some clusters, can be determined by first computing (co)limits of all the clusters, and then compute the (co)limits of the resulting (smaller) diagram (cf. \cite[\S 4.2.3]{HTT}).
\begin{theorem}[Limits by decomposing diagrams]\label{decomp-diagram}\index{limits decomposing}\label{lim-decomp}
	Let $K$ be a simplicial set and let $\varphi\colon K\to\operatorname{Cat}_{\infty}, i\mapsto\EuScript{C}(i)$ be a functor with colimit $\EuScript{C}$. Let $p\colon\EuScript{E}\to K$ be the cocartesian fibration classified by $\varphi$. Let $\EuScript{V}$ be an $\infty$-category, and let $F\colon\EuScript{C}\to\EuScript{V}$ be a functor. For each vertex $i\in K$, let $F|_{\EuScript{C}(i)}\colon\EuScript{C}(i)\to\EuScript{C}\xrightarrow{F}\EuScript{V}$ be a composition \myparen{in $\operatorname{Cat}_{\infty}$}.
	\begin{enumerate}[label=\myparen{\arabic*}]
		\item If $\EuScript{V}$ is cocomplete, we have a canonical equivalence
		\[
		\operatorname{colim}_{i\in K}\operatorname{colim}F|_{\EuScript{C}(i)}\xrightarrow{\simeq}\operatorname{colim}F.\]
		If $F$ is restricted from a functor $\EuScript{E}\to\EuScript{V}$, then the above colimit is also identified with $\operatorname{colim}(\EuScript{E}\to\EuScript{V})$.
		\item If $\EuScript{V}$ is complete, we have a canonical equivalence
		\[
		\operatorname{lim}F\xrightarrow{\simeq}\operatorname{lim}_{i\in K^{\mathrm{op}}}\operatorname{lim}F|_{\EuScript{C}(i)}.\]
		If $F$ is restricted from a functor $\EuScript{E}\to\EuScript{V}$, then the above limit is also identified with $\operatorname{lim}(\EuScript{E}\to\EuScript{V})$.
	\end{enumerate}
\end{theorem}
\begin{proof}
	\begin{enumerate}[label=(\arabic*)]
		\item We have a localisation functor $L\colon\EuScript{E}\to\EuScript{C}$, which is thus right cofinal. For each vertex $i\in K$, consider the following diagram
		\[
		\begin{tikzcd}
			\EuScript{C}(i) \ar[r, "{\rm cofinal}"] \ar[d]  \ar[dr, phantom, "\pb", very near start]   & \EuScript{E}_{/i} 		\ar[r] \ar[d]  \ar[dr, phantom, "\pb", very near start]   &  \EuScript{E}		\ar[r, "L"] \ar[d, "p"]   &  \EuScript{C}		\ar[r, "F"]		& \EuScript{V}	 \\
			\Delta^0  \ar[r, "{\rm cofinal}"]     & K_{/i} 		\ar[r]        & K \ar[urr, bend right=20, "p_!(F\circ L)"']  &
		\end{tikzcd}\]
		in $\sset$, where all three vertical arrows are cocartesian fibrations, and the bottom horizontal arrow $\Delta^0 \xrightarrow{1_i}K_{/i}$ is right cofinal (see \Cref{ex-rcofmaps}). So the arrow $\EuScript{C}(i)\to\EuScript{E}_{/i}$ is also right cofinal. We can compute the left Kan extension as
		\[p_!(F\circ L)(i)\simeq\operatorname{colim}F|_{\EuScript{C}(i)}.\]
		Thus we obtain
		\[
		\operatorname{colim}_{i\in K}\operatorname{colim}F|_{\EuScript{C}(i)}\simeq\operatorname{colim}p_!(F\circ L)\simeq\operatorname{colim}F\circ L\xrightarrow{\simeq}\operatorname{colim}F.\]
		\item Since $\operatorname{colim}(K\xrightarrow{\varphi}\operatorname{Cat}_{\infty}\xrightarrow{\mathrm{op}}\operatorname{Cat}_{\infty})\xrightarrow{\simeq}\EuScript{C}^{\mathrm{op}}$, we have a canonical equivalence
		\[
		\operatorname{colim}_{i\in K}\operatorname{colim}F^{\mathrm{op}}|_{\EuScript{C}(i)^{\mathrm{op}}}\xrightarrow{\simeq}\operatorname{colim}(\EuScript{C}^{\mathrm{op}}\xrightarrow{F^{\mathrm{op}}}\EuScript{V}^{\mathrm{op}})\]
		in $\EuScript{V}^{\mathrm{op}}$ by (1). Viewing in $\EuScript{V}$, we obtain
		\[\operatorname{lim}_{i\in K^{\mathrm{op}}}\operatorname{lim}F|_{\EuScript{C}(i)}\xleftarrow{\simeq}\operatorname{lim}F.\]
	\end{enumerate}
\end{proof}
\begin{prop}
	Let $\EuScript{C}$ be a small $\infty$-category, let $P$ be the filtered poset of full subcategories of $\EuScript{C}$ with finitely many objects \myparen{\emph{ordered by inclusion}}. Then the colimit of $A$ over $A\in P$ is canonically equivalent to $\EuScript{C}$.
\end{prop}
\begin{proof}
	We define a diagram $\varphi\colon P^{\triangleright}\to\operatorname{Cat}_{\infty}$ by sending $A\in P$ to $A$, and the cone point to $\EuScript{C}$ and verify \Cref{bas-crit-catcolim} (2). For every $\EuScript{E}\in\operatorname{Cat}_{\infty}$, by \Cref{spcat-adj} we have $\operatorname{Map}_{\operatorname{Cat}_{\infty}}(A, \EuScript{E})\simeq\operatorname{Map}_{\operatorname{Cat}_{\infty}}(A, \EuScript{E}^{\simeq})\simeq\operatorname{Map}_{\EuScript{S}}(|A|, \EuScript{E}^{\simeq})$, whose limit over $A\in P$ is canonically equivalent to $\operatorname{Map}_{\EuScript{S}}(|\EuScript{C}|, \EuScript{E}^{\simeq})\simeq\operatorname{Map}_{\operatorname{Cat}_{\infty}}(\EuScript{C}, \EuScript{E})$ (by \Cref{simpsubs-filt-colim-geomrel} below), i.e. $\operatorname{Map}_{\operatorname{Cat}_{\infty}}(\varphi, \EuScript{E})\colon(P^{\rm op})^{\triangleleft}\to\EuScript{S}$ is a limit diagram.
\end{proof}
\begin{remark}
	For an $\infty$-category $\EuScript{C}$, \cite[Corollary 2.33]{COJ} gives a criterion for a collection of (left-closed) full subcategories of $\EuScript{C}$ to have $\EuScript{C}$ as colimit.
\end{remark}

\section{Reshaping colimit diagrams into simplicial ones}

We can now give a non-technical proof the following result for colimits and limits of diagrams of the shape of the category of simplices of another simplicial set in the $\infty$-categorical setting, in parallel with \cite[Lemma 18.9.1]{Hir03} in the simplicial model-categorical setting. We take advantage of the above result on computing limits by decomposing diagrams.
\begin{prop}\label{lim-catsimp}
	Let $C$ be a simplicial set and let $\mathbf{\Delta}_{/C}$ be the category of simplices. Let $\EuScript{V}$ be an $\infty$-category.
	\begin{enumerate}[label=\myparen{\arabic*}]
		\item Assume that $\EuScript{V}$ is cocomplete. For any diagram $X\colon(\mathbf{\Delta}_{/C})^{\rm op}\to\EuScript{V}$, consider the new diagram $X'\colon\mathbf{\Delta}^{\rm op}\to\EuScript{V}, [n]\mapsto\coprod_{\sigma\in C_n}X(\sigma)$, we have a natural equivalence
		\[\operatorname{colim}((\mathbf{\Delta}_{/C})^{\rm op}\xrightarrow{X}\EuScript{V})\simeq\operatorname{colim}(\mathbf{\Delta}^{\rm op}\xrightarrow{X'}\EuScript{V}).\]
		\item Assume that $\EuScript{V}$ is complete. For any diagram $Y\colon\mathbf{\Delta}_{/C}\to\EuScript{V}$, consider the new diagram $Y'\colon\mathbf{\Delta}\to\EuScript{V}, [n]\mapsto\prod_{\sigma\in C_n}Y(\sigma)$, we have a natural equivalence
		\[\operatorname{lim}(\mathbf{\Delta}_{/C}\xrightarrow{Y}\EuScript{V})\simeq\operatorname{lim}(\mathbf{\Delta}\xrightarrow{Y'}\EuScript{V}).\]
		\item For any $K\in\EuScript{V}$, write $\underline{K}$ for the constant diagram with value $K$.  If $\EuScript{V}$ is \myparen{\emph{co}}complete, we have natural equivalences
		\[
		\begin{aligned}
			\operatorname{colim}((\mathbf{\Delta}_{/C})^{\rm op}\xrightarrow{\underline{K}}\EuScript{V})&\simeq\operatorname{colim}_{[n]\in\mathbf{\Delta}^{\rm op}}\coprod_{C_n}K,\\
			\operatorname{lim}(\mathbf{\Delta}_{/C}\xrightarrow{\underline{K}}\EuScript{V})&\simeq\operatorname{lim}_{[n]\in\mathbf{\Delta}}K^{C_n}.
		\end{aligned}\]
	\end{enumerate}
\end{prop}
\begin{proof}
	Apply \Cref{catsimp-cart} and \Cref{lim-decomp}.
\end{proof}
\begin{remark}
	Since $\mathbf{\Delta}_{\rm s}\hookrightarrow\mathbf{\Delta}$ is left cofinal and $\mathbf{\Delta}_{/C}\to\mathbf{\Delta}$ is a right fibration, $(\mathbf{\Delta}_{\rm s})_{/C}\hookrightarrow\mathbf{\Delta}_{/C}$ is also left cofinal. So in the above colimits and limits, we can safely replace $\mathbf{\Delta}_{/C}$ everywhere by $(\mathbf{\Delta}_{\rm s})_{/C}$ without changing the results; having this replacement will avoid a lot of redundancies.
\end{remark}
\begin{corollary}
	Let $C$ be a simplicial set and let $\mathbf{\Delta}_{/C}$ be its category of simplices. Let $K\in\EuScript{S}$, we have natural equivalences
	\begin{equation}\label{lim-cosimp}
		\begin{aligned}
			\operatorname{colim}((\mathbf{\Delta}_{/C})^{\rm op}\xrightarrow{\underline{K}}\EuScript{S})&\simeq\operatorname{colim}_{[n]\in\mathbf{\Delta}^{\rm op}}\coprod_{C_n}K,\\
			\operatorname{lim}(\mathbf{\Delta}_{/C}\xrightarrow{\underline{K}}\EuScript{S})&\simeq\operatorname{lim}_{[n]\in\mathbf{\Delta}}K^{C_n},\\
			\operatorname{colim}((\mathbf{\Delta}_{/C})^{\rm op}\xrightarrow{*}\EuScript{S})&\simeq\operatorname{colim}_{[n]\in\mathbf{\Delta}^{\rm op}}C_n.
		\end{aligned}
	\end{equation}
\end{corollary}
\begin{corollary}\label{spcat-adj}
	The inclusion functor $\EuScript{S}\hookrightarrow\operatorname{Cat}_{\infty}$ has a right adjoint $\operatorname{Core}\colon\operatorname{Cat}_{\infty}\to\EuScript{S}$ and a left adjoint $\operatorname{|-|}\colon\operatorname{Cat}_{\infty}\to\EuScript{S}$, given by inverting all arrows. Moreover, for any $\EuScript{C}\in\operatorname{Cat}_{\infty}$, we have
	\begin{equation}\label{htyp-qcat}
		|\EuScript{C}|\simeq\operatorname{colim}_{\EuScript{C}}*\simeq\operatorname{colim}((\mathbf{\Delta}_{/\EuScript{C}})^{\rm op}\xrightarrow{*}\EuScript{S})\simeq\operatorname{colim}_{[n]\in\mathbf{\Delta}^{\rm op}}\EuScript{C}_n\simeq|\EuScript{C}^{\rm op}|.
	\end{equation}
	Here the $\infty$-category $\mathbf{\Delta}_{/\EuScript{C}}$ is relative to the inclusion functor $\mathbf{\Delta}\hookrightarrow\operatorname{Cat}_{\infty}, [n]\mapsto\Delta^n$, and we view $\EuScript{C}_n$ as a discrete space.
\end{corollary}
\begin{proof}
	The first two descriptions of $|\EuScript{C}|$ is \cite[Corollary 2.10 (4) and (2)]{COJ}; the third then follows from (\ref{lim-cosimp}) above. We can also use \Cref{catsimp-cart}.
\end{proof}
\begin{prop}
	Let $K$ be a simplicial set and let $\varphi\colon K\to\operatorname{Cat}_{\infty}$ be a diagram classifying a cocartesian fibration $p\colon\EuScript{E}\to K$. Denote by $|\varphi|\colon K\xrightarrow{\varphi}\operatorname{Cat}_{\infty}\xrightarrow{\operatorname{|-|}}\EuScript{S}$ the composition.
\begin{enumerate}[label=\myparen{\arabic*}]
	\item We have $p_!(*)\simeq|\varphi|\colon K\to\operatorname{Cat}_{\infty}$, here $*$ denotes the constant functor $\EuScript{E}\to\EuScript{S}$ with value $*$.
	\item For $X, Y\in\operatorname{Fun}(K, \EuScript{S})$, we have
	\[\operatorname{lim}\left(\int_KX\xrightarrow{p_X}K\xrightarrow{Y}\EuScript{S}\right)\simeq\operatorname{Map}_{\operatorname{Fun}(K, \EuScript{S})}(X, Y),\]
	where $p_X\colon\int_KX\to K$ is the left fibration classified by $X$. In particular,
	\[\operatorname{lim}(K\xrightarrow{Y}\EuScript{S})\simeq\operatorname{Map}_{\operatorname{Fun}(K, \EuScript{S})}(*, Y).\]
\end{enumerate}
\end{prop}
\begin{proof}
	 \begin{enumerate}[label=(\arabic*)]
	 	\item We have $p_!(*)(i)\simeq\operatorname{colim}_{\EuScript{E}_{/i}}*\simeq|\EuScript{E}_{/i}|$, and a byproduct of proof of \Cref{lim-decomp} tells that $\varphi(i)\to\EuScript{E}_{/i}$ is right cofinal, so $|\EuScript{E}_{/i}|\simeq|\varphi(i)|$, and thus $p_!(*)\simeq|\varphi|$.
	 	\item Let $q\colon K\to\Delta^0$ be the projection, we have
	 	\[\operatorname{lim}(K\xrightarrow{X}\EuScript{S})\simeq q_*X\simeq\operatorname{Map}_{\EuScript{S}}(*, q_*X)\simeq\operatorname{Map}_{\operatorname{Fun}(K, \EuScript{S})}(*, X).\]
	 	Thus by adjunction,
	 	\[
	 	\operatorname{lim}p_X^*Y\simeq\operatorname{Map}_{\operatorname{Fun}(\int_KX, \EuScript{S})}(*, p_X^*Y)\simeq\operatorname{Map}_{\operatorname{Fun}(K, \EuScript{S})}((p_X)_!*, Y)\simeq\operatorname{Map}_{\operatorname{Fun}(K, \EuScript{S})}(X, Y).\]
 	\end{enumerate}
\end{proof}
\begin{prop}\label{colim-repble}
	Let $\EuScript{C, D}$ be $\infty$-categories with $\EuScript{C}$ small, let $p\colon\EuScript{C}\to\EuScript{D}$ be a functor. Then for any object $d\in\EuScript{D}$, we have
	\[\operatorname{colim}(p^*\operatorname{Map}_{\EuScript{D}}(-, d)\colon\EuScript{C}^{\rm op}\to\EuScript{S})\simeq|\EuScript{C}_{/d}|.
	\]
	
	In particular, if $\EuScript{D}$ is small \emph{(and non-empty)}, we have
	\[\operatorname{colim}(\operatorname{Map}_{\EuScript{D}}(-, d)\colon\EuScript{D}^{\rm op}\to\EuScript{S})\simeq*.
	\]
\end{prop}
\begin{proof}
	Since the functor $h_d\colon\EuScript{D}^{\rm op}\to\EuScript{S}$ classifies the right fibration $\EuScript{D}_{/d}\to\EuScript{D}$ and the unstraightening equivalences are compatible with base change, the functor $p^*\operatorname{Map}_{\EuScript{D}}(-, d)=h_d\circ p^{\rm op}\colon\EuScript{C}^{\rm op}\to\EuScript{S}, c\mapsto\operatorname{Map}_{\EuScript{D}}(p(c), d)$ classifies the right fibration $u\colon\EuScript{C}_{/d}\to\EuScript{C}$. Let $v\colon \EuScript{C}\to\Delta^0$ be the projection, then we have
	\[\operatorname{colim}_{\EuScript{C}^{\rm op}}(p^*\operatorname{Map}_{\EuScript{D}}(-, d))\simeq(v^{\rm op})_!(u^{\rm op})_!(*)\simeq\operatorname{colim}_{(\EuScript{C}_{/d})^{\rm op}}*\simeq|\EuScript{C}_{/d}|.
	\]
\end{proof}

For a simplicial set $C$, we also write $|C|:=|\EuScript{C}|\in\EuScript{S}$ for a Joyal fibrant model $C\to\EuScript{C}$ in $\sset$, and call it the \emph{homotopy type of the simplicial set $C$}, or the \emph{geometric realisation of $C$} (see \cite[Corollary 2.10 (6)]{COJ} for a justification) \index{homotopy type}.

We now extend the above result to all simplicial sets.
\begin{theorem}\label{htyp-sset}
	Let $C$ be a simplicial set, then we have
	\begin{equation}
		|C|\simeq\operatorname{colim}_{C}*\simeq\operatorname{colim}((\mathbf{\Delta}_{/C})^{\rm op}\xrightarrow{*}\EuScript{S})\simeq\operatorname{colim}_{[n]\in\mathbf{\Delta}^{\rm op}}C_n\simeq|C^{\rm op}|.
	\end{equation}
	We obtain a colimit diagram
	\newcommand{\stackspacesemi}{4}
	\newcommand{\stackk}[2][.8cm]{\;\tikz[baseline, yshift=.65ex]%
		{\foreach \k [evaluate=\k as \r using (.5*#2+.5-\k)*\stackspacesemi] in {1,...,#2}{%
				\k{\draw[->](0,\r pt)--(#1,\r pt);}
		}}\;}
	\begin{equation}
		\begin{tikzcd}
			\cdots\cdots\stackk{5}C_3\stackk{4}C_2\stackk{3}C_1\stackk{2}C_0 \arrow[r]   & {|C|}
		\end{tikzcd}
	\end{equation}
	in $\EuScript{S}$.
	
	Here each $C_n$ is viewed as a discrete space/category. The map $C_n\to|C|$ is given by $\sigma\mapsto\sigma(0)$, where for $\sigma\in C_n, 0\leqslant i\leqslant n$, we write $\sigma(i):=d_0\circ\cdots\circ\widehat{d_i}\circ\cdots\circ d_n(\sigma)$.
\end{theorem}
\begin{proof}
	We take a Joyal fibrant replacement $C\to\EuScript{C}$ in $\sset$. Since colimits are invariant upon replacing $C$ by categorically equivalent ones, we have $\operatorname{colim}_{C}*\xrightarrow{\simeq}\operatorname{colim}_{\EuScript{C}}*$.
	
	Now we use the model structure of marked simplicial sets $\sset^+$ as a model of $\infty$-categories. We have a canonical comparison map $\tau_C\colon\mathbf{\Delta}_{/C}\to C, (\Delta^n\xrightarrow{\sigma}C)\mapsto\sigma(n)$. We mark $C$ by the identities, and mark $\mathbf{\Delta}_{/C}$ by those $1$-simplices that are sent to the identities in $C$, then $\tau_C\colon\mathbf{\Delta}_{/C}\to C$ is an equivalence in the model structure of marked simplicial sets $\sset^+$, exhibiting $C$ as a localization of $\mathbf{\Delta}_{/C}$; see \cite[\href{https://kerodon.net/tag/01NC}{Kerodon tag/01NC}]{kerodon} or \cite[Proposition 7.3.15]{CiHi}. We see that $\EuScript{C}$ is a common localization of $\mathbf{\Delta}_{/C}$ and $\mathbf{\Delta}_{/\EuScript{C}}$, thus \[|C|=|\EuScript{C}|\simeq\operatorname{colim}_{\EuScript{C}}*\simeq\operatorname{colim}_{C}*\simeq\operatorname{colim}_{\EuScript{C}^{\rm op}}*\simeq\operatorname{colim}((\mathbf{\Delta}_{/C})^{\rm op}\xrightarrow{*}\EuScript{S}).\]
	The desired result now follows from (\ref{lim-cosimp}), noting that $(C^{\rm op})_n=C_n$.
\end{proof}
\begin{remark}
	It might be worth to note the above extreme phenomenon: \emph{the subcategory $\tau_{\leqslant 0}\EuScript{S}\simeq\set$ generates $\EuScript{S}$ under geometric realisation}, hence the inclusion $\tau_{\leqslant 0}\EuScript{S}\subset\EuScript{S}$ is far from being closed under colimits.
\end{remark}\begin{corollary}\label{simpsubs-filt-colim-geomrel}
	Let $C$ be a simplicial set, and let $P$ be a family of simplicial subsets of $C$ viewed as a poset under inclusion that is filtered and their \myparen{$1$-categorical} colimit in $\sset$ is $C$. Then the filtered colimit in $\EuScript{S}$ of $|A|$ over $A\in P$ is canonically equivalent to $|C|$.
\end{corollary}
\begin{proof}
	We use $|C|\simeq\operatorname{colim}_{[n]\in\mathbf{\Delta}^{\rm op}}C_n$, colimits commute with each other, and that the inclusion $\tau_{\leqslant 0}\EuScript{S}\hookrightarrow\EuScript{S}$ of discrete spaces into spaces commutes with filtered colimits.
\end{proof}
\begin{prop}
	Let $C$ be a simplicial set and let $K\in\EuScript{S}$, then
	\[\operatorname{Fun}(C, K)\simeq K^C\simeq\operatorname{lim}(C\xrightarrow{\underline{K}}\EuScript{S})\simeq\operatorname{lim}(\mathbf{\Delta}_{/C}\xrightarrow{\underline{K}}\EuScript{S})\simeq\operatorname{lim}_{[n]\in\mathbf{\Delta}}K^{C_n}.\]
\end{prop}
\begin{proof}
	Take a Joyal fibrant model $C\to\EuScript{C}$ in $\sset$, which is in particular a Kan-Quillen weak equivalence, so $K^{\EuScript{C}}\xrightarrow{\simeq}K^C$ in $\EuScript{S}$. Applying $\operatorname{Map}_{\EuScript{S}}(-, K)$ to the equivalences in \Cref{htyp-sset}, we use that $\operatorname{|-|}\colon\operatorname{Cat}_{\infty}\to\EuScript{S}$ is left adjoint to the inclusion to complete the proof.
\end{proof}
\begin{theorem}\label{colimits-simplicial}
	Let $C$ be a \myparen{small} simplicial set and let $\EuScript{E}$ be an $\infty$-category. Let $F\colon C\to\EuScript{E}$ be a diagram.
	
	For $\sigma\in C_n$ and $0\leqslant i\leqslant j\leqslant n$, we write $\sigma(i):=d_0\circ\cdots\circ\widehat{d_i}\circ\cdots\circ d_n(\sigma)\in C_0$, and write $(\sigma|_{[i, j]})_*$ for a chosen composite of the arrows $F(\sigma(i))\to F(\sigma(i+1))\to\cdots\to F(\sigma(j-1))\to F(\sigma(j))$ in $\EuScript{E}$.
	\begin{enumerate}[label=\myparen{\arabic*}]
		\item We have a natural equivalence
		\[\operatorname{lim}(C\xrightarrow{F}\EuScript{E})\simeq\operatorname{lim}_{[n]\in\mathbf{\Delta}}\prod_{\sigma\in C_n}F(\sigma(n))\]
		whenever all terms make sense, in which case, we have a limit diagram
		\newcommand{\stackspacesemi}{4}
		\newcommand{\stackk}[2][.8cm]{\;\tikz[baseline, yshift=.65ex]%
			{\foreach \k [evaluate=\k as \r using (.5*#2+.5-\k)*\stackspacesemi] in {1,...,#2}{%
					\k{\draw[->](0,\r pt)--(#1,\r pt);}
			}}\;}
		\[
		\begin{tikzcd}
			\operatorname{lim}(C\xrightarrow{F}\EuScript{E}) \arrow[r]   & \prod_{\sigma\in C_0}F(\sigma(0)) \stackk{2} \prod_{\sigma\in C_1}F(\sigma(1)) \stackk{3}\prod_{\sigma\in C_2}F(\sigma(2))\stackk{4}\cdots\cdots.
		\end{tikzcd}\]
		\item We have a natural equivalence
		\[\operatorname{colim}_{[n]\in\mathbf{\Delta}^{\rm op}}\coprod_{\sigma\in C_n}F(\sigma(0))\simeq\operatorname{colim}(C\xrightarrow{F}\EuScript{E})\]
		whenever all terms make sense, in which case, we have a colimit diagram
		\[
		\begin{tikzcd}
			\cdots\cdots\stackk{4}\coprod_{\sigma\in C_2}F(\sigma(0))\stackk{3}\coprod_{\sigma\in C_1}F(\sigma(0)) \stackk{2}\coprod_{\sigma\in C_0}F(\sigma(0)) \arrow[r]   & \operatorname{colim}(C\xrightarrow{F}\EuScript{E}).
		\end{tikzcd}\]
	\end{enumerate}
	
	Here, for a map $\alpha\colon[m]\to[n]$ in $\mathbf{\Delta}$, the relevant \myparen{\emph{co}}simplicial structure maps are those fit into the commutative diagrams
	\[\begin{tikzcd}[column sep=3em]
		\prod_{\sigma\in C_m}F(\sigma(m)) \ar[r, "\alpha_*"] \ar[d]  & \prod_{\rho\in C_n}F(\rho(n))  \arrow[d]\\
		F(\alpha^*\rho'(m)) \ar[r, "(\rho'|_{[\alpha(m), n]})_*"']      & F(\rho'(n))
	\end{tikzcd}
	\qquad
	\textrm{and}
	\qquad
	\begin{tikzcd}[column sep=3em]
		F(\rho'(0)) \ar[r, "(\rho'|_{[0, \alpha(0)]})_*"] \ar[d]  & F(\alpha^*\rho'(0))  \arrow[d]\\
		\coprod_{\rho\in C_n}F(\rho(0)) \ar[r, "\alpha^*"']      & \coprod_{\sigma\in C_m}F(\sigma(0))
	\end{tikzcd}\]
	for every $\rho'\in C_n$.
\end{theorem}
\begin{proof}
	The statement (1) follows from (2) is by taking opposite. We only prove (2). Since colimits in $\EuScript{E}$, exist or not, can be tested on applying $\operatorname{Map}_{\EuScript{E}}(-, e)$ for objects $e\in\EuScript{E}$, we can even assume $\EuScript{E}=\EuScript{S}$, with the merit of being bicomplete.
	
	In this case, we left Kan extend $F$ to a functor $|C|\to\EuScript{S}$, which is just restriction of $F$, and they share the same colimit. By \Cref{htyp-sset} and \Cref{lim-decomp}, $\operatorname{colim}(|C|\to\EuScript{S})$ is identified with $\operatorname{colim}_{[n]\in\mathbf{\Delta}^{\rm op}}\operatorname{colim}F|_{C_n}$. As $\operatorname{colim}F|_{C_n}\simeq\coprod_{\sigma\in C_n}F(\sigma(0))$, we are done.
\end{proof}
\begin{remark}
	\begin{enumerate}[label=(\arabic*)]
		\item This recovers the familiar result of writing a (co)limit as a (co)equaliser of (co)products in ordinary categories, in view of \Cref{$n$-cofinal-lim-pres} and \Cref{simp-ncofinal} below.
		\item If $C$ is (the nerve of) the category of simplices of a simplicial set, the above result recovers \Cref{lim-catsimp}, once we discard the correct redundant terms (a first step is to use $\mathbf{\Delta}_{\rm s}$ instead of $\mathbf{\Delta}$).
	\end{enumerate}
\end{remark}

\section{$n$-cofinality}

In this section, we discuss $n$-cofinality, which will be a good replacement for cofinality in the setting of $(n, 1)$-categories, and give a few characterisations for it. Some examples of $n$-cofinal functors will be given.
\begin{definition}[$n$-cofinality]\label{$n$-cofinality}
	Let $v\colon K'\to K$ be a map of (small) simplicial sets with $K$ an $\infty$-category, let $-2\leqslant n\leqslant\infty$. We say that $v\colon K'\to K$ is \emph{right $n$-cofinal} if for every object $x\in K$, the simplicial set $K'_{x/}$ is $n$-connective.
	
	We say that $v\colon K'\to K$ is \emph{left $n$-cofinal} if $v^{\rm op}\colon K'^{\rm op}\to K^{\rm op}$ is right $n$-cofinal, i.e., for every object $x\in K$, the simplicial set $K'_{/x}$ is $n$-connective.
\end{definition}

If $v\colon K'\to K$ is left/right $n$-cofinal, and $K=\varnothing$, then $K'=\varnothing$; if $K\ne\varnothing$, then $K'\ne\varnothing$.

The following result characterising $n$-cofinality is a generalisation of \cite[Theorem 6.4.5]{CiHi} and \cite[Theorem 2.19]{COJ} with essentially the same proof (for (1)(2)(3), at least). Here we use the notion of $n$-cofinality from \Cref{$n$-cofinality} and recall that an $(n, 1)$-category is an $\infty$-category with all mapping spaces $(n-1)$-truncated.
\begin{theorem}\label{$n$-cofinal-lim-pres}
	Let $\EuScript{C, D}$ be small $\infty$-categories, let $p\colon\EuScript{C}\to\EuScript{D}$ be a functor. Let $-1\leqslant n\leqslant\infty$ and let $\kappa$ be a regular cardinal such that $\EuScript{C}$ and $\EuScript{D}$ are $\kappa$-small. Then the following are equivalent.
	\begin{enumerate}[label=\myparen{\arabic*}]
		\item The functor $p\colon\EuScript{C}\to\EuScript{D}$ is left $n$-cofinal: for every object $d\in\EuScript{D}$, the simplicial set $\EuScript{C}_{/d}$ is $n$-connective.
		\item The functor $p$ respects limits in $(n, 1)$-categories: for every $(n, 1)$-category $\EuScript{E}$ and every functor $F\colon\EuScript{D}\to\EuScript{E}$, the comparison map of limits
		\[\operatorname{lim}(\EuScript{D}\xrightarrow{F}\EuScript{E})\to\operatorname{lim}(\EuScript{C}\xrightarrow{p}\EuScript{D}\xrightarrow{F}\EuScript{E})\]
		is an equivalence \myparen{in the sense that if one limit exists, so does the other, and the map is an equivalence}. Alternatively, for every limit diagram $\overline{F}\colon \EuScript{D}^{\triangleleft}\to\EuScript{E}$, the induced diagram $\EuScript{C}^{\triangleleft}\xrightarrow{p^{\triangleleft}} \EuScript{D}^{\triangleleft}\xrightarrow{\overline{F}}\EuScript{E}$ is also a limit diagram.
		\item The functor $p$ respects limits in $\tau_{\leqslant n-1}\EuScript{S}$: for every functor $F\colon\EuScript{D}\to\tau_{\leqslant n-1}\EuScript{S}$, the comparison map of limits
		\[\operatorname{lim}(\EuScript{D}\xrightarrow{F}\tau_{\leqslant n-1}\EuScript{S})\to\operatorname{lim}(\EuScript{C}\xrightarrow{p}\EuScript{D}\xrightarrow{F}\tau_{\leqslant n-1}\EuScript{S})\]
		is an equivalence.
		\item The functor $p^{\rm op}\colon\EuScript{C}^{\rm op}\to\EuScript{D}^{\rm op}$ respects colimits in $(n, 1)$-categories: for every $(n, 1)$-category $\EuScript{E}$ and every functor $F\colon\EuScript{D}^{\rm op}\to\EuScript{E}$, the comparison map of colimits
		\[\operatorname{colim}(\EuScript{C}^{\rm op}\xrightarrow{p^{\rm op}}\EuScript{D}^{\rm op}\xrightarrow{F}\EuScript{E})\to\operatorname{colim}(\EuScript{D}^{\rm op}\xrightarrow{F}\EuScript{E})\]
		is an equivalence \myparen{in the sense that if one colimit exists, so does the other, and the map is an equivalence}. Alternatively, for every colimit diagram $\overline{F}\colon(\EuScript{D}^{\rm op})^{\triangleright}\to\EuScript{E}$, the induced diagram $(\EuScript{C}^{\rm op})^{\triangleright}\xrightarrow{(p^{\rm op})^{\triangleright}} (\EuScript{D}^{\rm op})^{\triangleright}\xrightarrow{\overline{F}}\EuScript{E}$ is also a colimit diagram.
		\item The functor $p^{\rm op}\colon\EuScript{C}^{\rm op}\to\EuScript{D}^{\rm op}$ respects colimits in $\operatorname{Ind}_{\kappa}(\EuScript{E})$ for every small $(n, 1)$-category $\EuScript{E}$.
		\item The functor $p^{\rm op}\colon\EuScript{C}^{\rm op}\to\EuScript{D}^{\rm op}$ respects colimits in $\tau_{\leqslant n-1}\EuScript{P}(\EuScript{E})$ for every small $(n, 1)$-category $\EuScript{E}$: for every functor $F\colon\EuScript{D}^{\rm op}\to\EuScript{P}(\EuScript{E})$, the comparison map of colimits
		\[\operatorname{colim}(\EuScript{C}^{\rm op}\xrightarrow{p^{\rm op}}\EuScript{D}^{\rm op}\xrightarrow{F}\EuScript{P}(\EuScript{E}))\to\operatorname{colim}(\EuScript{D}^{\rm op}\xrightarrow{F}\EuScript{P}(\EuScript{E}))\]
		is $n$-connective.
		\item The functor $p^{\rm op}\colon\EuScript{C}^{\rm op}\to\EuScript{D}^{\rm op}$ respects colimits in every $n$-topos \myparen{an $\infty$-category equivalent to one of the form $\tau_{\leqslant n-1}\EuScript{X}$ for an $\infty$-topos $\EuScript{X}$}.
		\item The functor $p^{\rm op}\colon\EuScript{C}^{\rm op}\to\EuScript{D}^{\rm op}$ respects colimits in $\tau_{\leqslant n-1}\EuScript{S}$: for every functor $F\colon\EuScript{D}^{\rm op}\to\tau_{\leqslant n-1}\EuScript{S}$, the comparison map of colimits
		\[\operatorname{colim}(\EuScript{C}^{\rm op}\xrightarrow{p^{\rm op}}\EuScript{D}^{\rm op}\xrightarrow{F}\tau_{\leqslant n-1}\EuScript{S})\to\operatorname{colim}(\EuScript{D}^{\rm op}\xrightarrow{F}\tau_{\leqslant n-1}\EuScript{S})\]
		is an equivalence.
		\item For every object $d\in\EuScript{D}$, the comparison map of colimits
		\[\operatorname{colim}(\EuScript{C}^{\rm op}\xrightarrow{p^{\rm op}}\EuScript{D}^{\rm op}\xrightarrow{\tau_{\leqslant n-1}\operatorname{Map}_{\EuScript{D}}(-, d)}\tau_{\leqslant n-1}\EuScript{S})\to\operatorname{colim}(\EuScript{D}^{\rm op}\xrightarrow{\tau_{\leqslant n-1}\operatorname{Map}_{\EuScript{D}}(-, d)}\tau_{\leqslant n-1}\EuScript{S})\]
		is an equivalence.
	\end{enumerate}
\end{theorem}
We start with some remarks that will be useful in the proof. Let $q\colon\EuScript{D}\to\Delta^0$ be the projection, then $\operatorname{lim}(\EuScript{D}\xrightarrow{F}\EuScript{E})\simeq q_*F, \operatorname{lim}(\EuScript{C}\xrightarrow{p}\EuScript{D}\xrightarrow{F}\EuScript{E})\simeq (q\circ p)_*(F\circ p)$. If $\EuScript{E}$ is complete, the comparison map of limits can be identified with the map $q_*F\to q_*(p_*p^*)F$ induced by the unit of the adjunction $p^*\dashv p_*$.

Moreover, for every functor $F\colon \EuScript{D}\to\tau_{\leqslant n-1}\EuScript{S}$,
\[\operatorname{lim}(\EuScript{D}\xrightarrow{F}\tau_{\leqslant n-1}\EuScript{S})\simeq\operatorname{Map}_{\operatorname{Fun}(\EuScript{D}, \tau_{\leqslant n-1}\EuScript{S})}(*, F).\]
\begin{proof}
	Since limits in $\EuScript{E}$, exist or not, can be tested on applying $\operatorname{Map}_{\EuScript{E}}(e,-)$ for objects $e\in\EuScript{E}$ (note that $\tau_{\leqslant n-1}\EuScript{S}$ is closed under limits in $\EuScript{S}$), (2) and (3) are equivalent. Since an $\infty$-category $\EuScript{E}$ is an $(n, 1)$-category if and only if its opposite $\EuScript{E}^{\rm op}$ is an $(n, 1)$-category, (2) and (4) are equivalent.
	
	By the previous remark and adjunction, (3) is equivalent to
	\[\operatorname{Map}_{\operatorname{Fun}(\EuScript{D}, \tau_{\leqslant n-1}\EuScript{S})}(*, F)\xrightarrow[p^*]{\sim}\operatorname{Map}_{\operatorname{Fun}(\EuScript{C}, \EuScript{S})}(*, F\circ p)\simeq\operatorname{Map}_{\operatorname{Fun}(\EuScript{C}, \EuScript{S})}(*, p^*F)\simeq\operatorname{Map}_{\operatorname{Fun}(\EuScript{D}, \tau_{\leqslant n-1}\EuScript{S})}(p_!(*), F)\]
	for all $F\colon\EuScript{D}\to\tau_{\leqslant n-1}\EuScript{S}$.
	
	This is the case if and only if the (unique) map $p_!(*)\to *$ is an equivalence in $\operatorname{Fun}(\EuScript{D}, \tau_{\leqslant n-1}\EuScript{S})$. Since $p_!(*)(d)\simeq\operatorname{colim}_{\EuScript{C}_{/d}}*\simeq\operatorname{colim}_{\EuScript{C}_{/d}} \tau_{\leqslant n-1}*\simeq \tau_{\leqslant n-1}\operatorname{colim}_{\EuScript{C}_{/d}}*\simeq \tau_{\leqslant n-1}|\EuScript{C}_{/d}|\in\tau_{\leqslant n-1}\EuScript{S}$, (3) holds if and only if for every object $d\in\EuScript{D}$, $ \tau_{\leqslant n-1}|\EuScript{C}_{/d}|\simeq*$ in $\tau_{\leqslant n-1}\EuScript{S}$, i.e. $\EuScript{C}_{/d}$ is $n$-connective. So (1) and (3) are equivalent.
	
	The implication $(6)\Rightarrow(5)\Rightarrow(4)$ can be proved along the same line as the proof of \cite[Proposition A.1]{HP22}, where we need \cite[Proposition 5.3.5.14]{HTT} (or \Cref{colim-yon} (2) below) for $(5)\Rightarrow(4)$. We conclude by noting $(4)\Rightarrow(7)\Rightarrow(6)$ and $(4)\Rightarrow(8)\Rightarrow(9)\Rightarrow(1)$, where we apply \Cref{colim-repble} for the last implication.
\end{proof}
\begin{remark}
	By possibly enlarging universe, we may assume that the $(n, 1)$-categories $\EuScript{E}$ appear above are small. Then $\EuScript{P}(\EuScript{E})$ and $\operatorname{Ind}_{\kappa}(\EuScript{E})$ are presentable.
	
	The formulation of (4)(5)(6) are of course inspired by the proof of \cite[Proposition A.1]{HP22}. The statements of (1)(2)(4) are most general, while the condition (3)(9) are (arguably) easy to verify in practice.
\end{remark}
\begin{remark}\label{cofintopt-cst}
Since (co)limits, and more generally Kan extensions, are invariant upon replacing the source simplicial sets by categorically equivalent ones, by taking a functorial Joyal fibrant replacement (and suitable Kan extensions along them), we can extend the notion of $n$-cofinality, as well as \Cref{$n$-cofinal-lim-pres} to \emph{all maps of simplicial sets}. In particular, for a simplicial set $K$, the map $K\to\Delta^0$ is left/right $n$-cofinal if and only if $K\ne\varnothing$ and $K$ is $n$-connective.

If $\EuScript{E}$ is an $(n, 1)$-category and $K$ is (non-empty and) $n$-connective, the (co)limit of a constant functor $K\to\Delta^0\xrightarrow{a}\EuScript{E}$ is given by $a\in\EuScript{E}$. It is also easy to see that $\operatorname{colim}(K\xrightarrow{*}\tau_{\leqslant n-1}\EuScript{S})\simeq\tau_{\leqslant n-1}|K|$, so if $\operatorname{colim}(K\xrightarrow{*}\tau_{\leqslant n-1}\EuScript{S})\simeq*$, then $K$ is $n$-connective --- a non-empty simplicial set $K$ is $n$-connective if and only if the functor $\operatorname{colim}\colon\operatorname{Fun}(K, \tau_{\leqslant n-1}\EuScript{S})\to\tau_{\leqslant n-1}\EuScript{S}$ preserves the terminal object (empty product); cf. \Cref{sift-finprod}.

As $(n, 1)$-categories are $(n+1, 1)$-categories, \emph{left $(n+1)$-cofinal maps are left $n$-cofinal}.
\end{remark}
\begin{corollary}
	Let $f\colon A\to B$ be a right $n$-cofinal map in $\sset$, $-1\leqslant n\leqslant\infty$. Then $|A|$ is $n$-connective if and only if $|B|$ is $n$-connective.
\end{corollary}
\begin{proof}
	By (simplicial set version of) \Cref{$n$-cofinal-lim-pres} (6), $\tau_{\leqslant n-1}|f|\colon\tau_{\leqslant n-1}|A|\to\tau_{\leqslant n-1}|B|$ is an equivalence in $\tau_{\leqslant n-1}\EuScript{S}$.
\end{proof}
\begin{corollary}\label{ncof-comp}
	Let $A\xrightarrow{f}B\xrightarrow{g}C$ be maps of simplicial sets, let $-1\leqslant n\leqslant\infty$.
	\begin{enumerate}[label=\myparen{\arabic*}]
		\item Assume that $f$ is left $n$-cofinal. Then $g$ is left $n$-cofinal if and only if the composite map $g\circ f$ is left $n$-cofinal. In particular, left $n$-cofinal maps of simplicial sets are closed under composition.
		\item Assume that $g$ is a full embedding of $\infty$-categories. If $g\circ f$ is left $n$-cofinal,
		then so are $f$ and $g$.
	\end{enumerate}
\end{corollary}
\begin{proof}
	\begin{enumerate}[label=(\arabic*)]
		\item This follows easily from (simplicial set version of) \Cref{$n$-cofinal-lim-pres} (3).
		\item Let $F\colon B\to\tau_{\leqslant n-1}\EuScript{S}$ be any diagram, we will show $\operatorname{lim}F\xrightarrow{\simeq}\operatorname{lim}f^*F$. As $g\circ f$ is left $n$-cofinal and $g$ is a full embedding of $\infty$-categories, we have
		\[
		\operatorname{lim}F\simeq\operatorname{lim}g_*F\simeq\operatorname{lim}(g\circ f)^*g_*F\simeq\operatorname{lim}f^*g^*g_*F\simeq\operatorname{lim}f^*F.\]
	\end{enumerate}
\end{proof}
\begin{prop}
	Given pullback squares in $\sset$ of the form
	\[
	\begin{tikzcd}
		A'' \ar[r, "u"]    \ar[d, "p''"']   \arrow[dr, phantom, "\pb", very near start]  & A' \ar[r] \ar[d, "p'"'] \arrow[dr, phantom, "\pb", very near start] & A \arrow[d, "p"]\\
		B''  \ar[r, "v"]    & B' \ar[r]      & B,
	\end{tikzcd}\]
	with $p\colon A\to B$ a cartesian fibration. If $v\colon B''\to B'$ is left $n$-cofinal, then so is $u\colon A''\to A'$.
	
	In particular, left $n$-cofinal maps of simplicial sets are stable under pullback along cartesian fibrations.
\end{prop}
\begin{proof}
	We will apply (simplicial set version of) \Cref{$n$-cofinal-lim-pres} (3). So let $F\colon A'\to\tau_{\leqslant n-1}\EuScript{S}$ be any diagram, we need to show
	\[\operatorname{lim}F\xrightarrow{\simeq}\operatorname{lim}u^*F.\]
	
	Since $p, p', p''$ are cartesian fibrations, we use again the byproduct of proof of \Cref{lim-decomp} to obtain that $p''^{-1}(d)\to A''_{d/}$ is left cofinal for every vertex $d\in B''$. Moreover, $u$ induces an isomorphism $p''^{-1}(d)\xrightarrow{\cong}p'^{-1}(v(d))$. As $v\colon B''\to B'$ is left $n$-cofinal, we have
	\[
	\begin{aligned}
		\operatorname{lim}F&\simeq\operatorname{lim}p'_*F\simeq\operatorname{lim}v^*p'_*F\simeq\operatorname{lim}_{d\in B''}(p'_*F)(v(d))\simeq\operatorname{lim}_{d\in B''}\operatorname{lim}(F|_{p'^{-1}(v(d))})\\
		&\simeq\operatorname{lim}_{d\in B''}\operatorname{lim}(F\circ u|_{p''^{-1}(d)})\simeq\operatorname{lim}_{d\in B''}p''_*(u^*F)(d)\simeq\operatorname{lim}u^*F,
	\end{aligned}\]
	as desired.
\end{proof}
\begin{remark}
	The proof indicates that it is sufficient for the weaker condition like $p^{-1}(b)\to A_{b/}$ is left $n$-cofinal to hold for every vertex $b\in B$ (rather than requiring $p$ to be a cartesian fibration). This should be another way to characterise all maps of simplicial sets sharing such pullback stability property (as $p$) in this proposition.
\end{remark}
In the same spirit, one can easily verify the following result (using that limits commute with limits).
\begin{prop}
	If $f\colon A\to B, f'\colon A'\to B'$ are left $n$-cofinal maps in $\sset$, then so is $f\times f'\colon A\times A'\to B\times B'$.
\end{prop}
Here are two variants of the previous proposition, which in fact generalise it (by taking $A_0=B_0=\Delta^0$); we cannot give an as easy proof however.
\begin{prop}
	Given a commutative diagram
	\[\begin{tikzcd}[column sep=3em]
		A_1 \ar[r, "\alpha_1"] \ar[d, "v_1"']  &  A_0 \ar[d, "v_0"']  &  A_2  \ar[l, "\alpha_2"']  \arrow[d, "v_2"]\\
		B_1  \ar[r, "\beta_1"]   &  B_0    &  B_2  \ar[l, "\beta_2"'] 
	\end{tikzcd}\]
	in $\sset$, with horizontal arrows cocartesian fibration, and assume that $v_0$ is right $n$-cofinal, $-1\leqslant n\leqslant\infty$. If for each vertex $a_0\in A_0$, the restrictions $v_1\colon\alpha_1^{-1}(a_0)\to\beta_1^{-1}(v_0(a_0))$ and $v_2\colon\alpha_2^{-1}(a_0)\to\beta_2^{-1}(v_0(a_0))$ are right $n$-cofinal, then the induced map $v=v_1\times v_2\colon A_1\times_{A_0}A_2\to B_1\times_{B_0}B_2$ is right $n$-cofinal.
\end{prop}
\begin{proof}
	By assumption, the projections $\alpha=\alpha_1\times\alpha_2\colon A_1\times_{A_0}A_2\to A_0$ and $\beta=\beta_1\times\beta_2\colon B_1\times_{B_0}B_2\to B_0$ are cocartesian fibration, and the restriction $v\colon\alpha^{-1}(a_0)\to\beta^{-1}(v_0(a_0))$ is right $n$-cofinal for each vertex $a_0\in A_0$. Thus for every $n$-topos $\EuScript{X}$ and every diagram $F\colon B_1\times_{B_0}B_2\to\EuScript{X}$, by \Cref{lim-decomp} we have
	\[
	\begin{aligned}
		\operatorname{colim}F&\simeq\operatorname{colim}_{b_0\in B_0}\operatorname{colim}F|_{\beta^{-1}(b_0)}\simeq\operatorname{colim}_{a_0\in A_0}\operatorname{colim}F|_{\beta^{-1}(v_0(a_0))}\\
		&\simeq\operatorname{colim}_{a_0\in A_0}\operatorname{colim}F\circ v|_{\alpha^{-1}(a_0)}\simeq	\operatorname{colim}v^*F.
	\end{aligned}\]
\end{proof}
\begin{prop}\label{lrcofinal-cocart}
Given a commutative diagram
\[\begin{tikzcd}[column sep=3em]
	A_1 \ar[r, "\alpha_1"] \ar[d, "v_1"']  &  A_0 \ar[d, "v_0"']  &  A_2  \ar[l, "\alpha_2"']  \arrow[d, "v_2"]\\
	B_1  \ar[r, "\beta_1"]   &  B_0    &  B_2  \ar[l, "\beta_2"'] 
\end{tikzcd}\]
in $\sset$, with horizontal arrows cocartesian fibration, and assume that $v_0$ is right $n$-cofinal, $-1\leqslant n\leqslant\infty$. If for each vertex $a_0\in A_0$, the restrictions $v_1\colon\alpha_1^{-1}(a_0)\to\beta_1^{-1}(v_0(a_0))$ and $v_2\colon\alpha_2^{-1}(a_0)\to\beta_2^{-1}(v_0(a_0))$ are left $n$-cofinal, then the induced map $v=v_1\times v_2\colon A_1\times_{A_0}A_2\to B_1\times_{B_0}B_2$ is left $n$-cofinal.
\end{prop}
\begin{proof}
	By assumption, the projections $\alpha=\alpha_1\times\alpha_2\colon A_1\times_{A_0}A_2\to A_0$ and $\beta=\beta_1\times\beta_2\colon B_1\times_{B_0}B_2\to B_0$ are cocartesian fibration, and the restriction $v\colon\alpha^{-1}(a_0)\to\beta^{-1}(v_0(a_0))$ is left $n$-cofinal for each vertex $a_0\in A_0^{\rm op}$. Thus for every diagram $F\colon B_1\times_{B_0}B_2\to\tau_{\leqslant n-1}\EuScript{S}$, we have
	\[
	\begin{aligned}
		\operatorname{lim}F&\simeq\operatorname{lim}_{b_0\in B_0^{\rm op}}\operatorname{lim}F|_{\beta^{-1}(b_0)}\simeq\operatorname{lim}_{a_0\in A_0^{\rm op}}\operatorname{lim}F|_{\beta^{-1}(v_0(a_0))}\\
		&\simeq\operatorname{lim}_{a_0\in A_0^{\rm op}}\operatorname{lim}F\circ v|_{\alpha^{-1}(a_0)}\simeq	\operatorname{lim}v^*F.
	\end{aligned}\]
\end{proof}
The following result generalises \cite[Lemma 6.5.3.10]{HTT}.
\begin{prop}\label{cofinal-ff-colim}
	Let $\EuScript{C, D}$ be small $\infty$-categories, let $p\colon\EuScript{C}\to\EuScript{D}$ be a functor. Let $-1\leqslant n\leqslant\infty$.
	\begin{enumerate}[label=\myparen{\arabic*}]
		\item Assume that $\EuScript{E}$ is a cocomplete $(n, 1)$-category. Let $\alpha\colon U\to V$ be an edge in $\operatorname{Fun}(\EuScript{D}, \EuScript{E})$. If $p\colon\EuScript{C}\to\EuScript{D}$ is right $n$-cofinal, then the induced morphism $\operatorname{colim}p^*\alpha\colon\operatorname{colim}p^*U\to\operatorname{colim}p^*V$ is an equivalence \myparen{resp. \emph{$n$-connective}} if and only if $\operatorname{colim}\alpha\colon\operatorname{colim}U\to\operatorname{colim}V$ is an equivalence \myparen{resp. \emph{$n$-connective}}.
		\item Assume that the functor $p\colon\EuScript{C}\to\EuScript{D}$ is fully faithful. If for every $n$-topos $\EuScript{X}$ and every edge $\alpha\colon U\to V$ in $\operatorname{Fun}(\EuScript{D}, \EuScript{X})$,
		\[p^*\alpha\colon p^*U\xrightarrow{\simeq} p^*V\Longrightarrow\operatorname{colim}\alpha\colon\operatorname{colim}U\xrightarrow{\simeq}\operatorname{colim}V,\]
		then $p\colon\EuScript{C}\to\EuScript{D}$ is right $n$-cofinal.
	\end{enumerate}
\end{prop}
\begin{proof}
	\begin{enumerate}[label=(\arabic*)]
		\item Consider the following commutative diagram:
		\[\begin{tikzcd}[column sep=4em]
			\operatorname{colim}p^*U \ar[r, "\operatorname{colim}p^*\alpha"] \ar[d]  &  \operatorname{colim}p^*V  \arrow[d]\\
			\operatorname{colim}U \ar[r, "\operatorname{colim}\alpha"]      &  \operatorname{colim}V.
		\end{tikzcd}\]
		As $p\colon\EuScript{C}\to\EuScript{D}$ is right $n$-cofinal and $\EuScript{E}$ is an $(n, 1)$-category, the two vertical arrows are equivalences by \Cref{$n$-cofinal-lim-pres} (4). The result follows.
		\item For a functor $V\colon\EuScript{D}\to\EuScript{X}$, we let $U:=p_!p^*V\colon\EuScript{D}^{\rm op}\to\EuScript{X}$ and let $\alpha\colon U\to V$ be the counit. Then $p^*U=p^*p_!p^*V\simeq p^*V$ (left Kan extension of a functor along a fully faithful functor restricts back to the original functor), so $\operatorname{colim}p^*U\simeq\operatorname{colim}p^*V\simeq\operatorname{colim}U$. And the assumption implies that the bottom arrow $\operatorname{colim}\alpha$ is an equivalence. Thus the right vertical comparison map $\operatorname{colim}p^*V\simeq\operatorname{colim}V$ is also an equivalence. This holds for every $n$-topos $\EuScript{X}$ and every functor $V\colon\EuScript{D}\to\EuScript{X}$, the result now follows from \Cref{$n$-cofinal-lim-pres} (7).
	\end{enumerate}
\end{proof}
\begin{corollary}\label{simp-ncofinal}
	For every $n\in\mathbb{N}$, the inclusion functor $(\mathbf{\Delta}^{\leqslant n})^{\rm op}\hookrightarrow\mathbf{\Delta}^{\rm op}$ is right $n$-cofinal.
\end{corollary}
\begin{proof}
	Immediately from \Cref{cofinal-ff-colim} (2) and \cite[Lemma 6.5.3.10]{HTT}.
\end{proof}
\begin{example}\label{lcof-simp-inj-fin}
\emph{For $0\leqslant n\leqslant m\leqslant\infty$, the inclusion functor $\mathbf{\Delta}^{\leqslant n}\hookrightarrow\mathbf{\Delta}^{\leqslant m}$ is left $n$-cofinal, as are $\mathbf{\Delta}_{\rm s}^{\leqslant n}\hookrightarrow\mathbf{\Delta}_{\rm s}^{\leqslant m}$ and $\mathbf{\Delta}_{\rm s}^{\leqslant n}\hookrightarrow\mathbf{\Delta}^{\leqslant n}$.}

The first statement is confirmed by \Cref{simp-ncofinal} (which by \Cref{ncof-comp} (2), can assume $m=\infty$). The second follows from the first and that the leftmost vertical arrow in the commutative diagram (\ref{cof-simpidx}) is right cofinal hence a weak homotopy equivalence (so $(\mathbf{\Delta}_{\rm s}^{\leqslant n})_{/[m]}$ and $(\mathbf{\Delta}^{\leqslant n})_{/[m]}$ share the same connectivity). For the last one, we need to use \Cref{$n$-cofinal-lim-pres} (3). For any functor $F\colon\mathbf{\Delta}^{\leqslant n}\to\tau_{\leqslant n-1}\EuScript{S}$, we denote by $G\colon\mathbf{\Delta}\to\tau_{\leqslant n-1}\EuScript{S}$ a right Kan extension of $F$ along the fully faithful inclusion $\mathbf{\Delta}^{\leqslant n}\hookrightarrow\mathbf{\Delta}$. Then we have
\[\operatorname{lim}F\simeq\operatorname{lim}G\simeq\operatorname{lim}_{\mathbf{\Delta}_{\rm s}}G|_{\mathbf{\Delta}_{\rm s}}\simeq\operatorname{lim}_{\mathbf{\Delta}_{\rm s}^{\leqslant n}}G|_{\mathbf{\Delta}_{\rm s}^{\leqslant n}}=\operatorname{lim}(\mathbf{\Delta}_{\rm s}^{\leqslant n}\hookrightarrow\mathbf{\Delta}^{\leqslant n}\xrightarrow{F}\tau_{\leqslant n-1}\EuScript{S}),\]
as desired (we have used that $\mathbf{\Delta}_{\rm s}\hookrightarrow\mathbf{\Delta}$ is left cofinal and $\mathbf{\Delta}_{\rm s}^{\leqslant n}\hookrightarrow\mathbf{\Delta}_{\rm s}$ is left $n$-cofinal for the last two equivalences); we can also directly see this from \Cref{ncof-comp} (2).
\end{example}
\begin{remark}\label{lim-trunc-des}
	This, together with \Cref{colimits-simplicial}, explains why Čech descent for sheaves of spaces needs to take limits of the whole $\mathbf{\Delta}_{\rm s}$-shaped diagrams, while for sheaves of sets, we only need to take equalisers (limits of $\mathbf{\Delta}_{\rm s}^{\leqslant 1}$-shaped diagrams). Similarly, the usual stack condition involves limits of $\mathbf{\Delta}_{\rm s}^{\leqslant 2}$-shaped diagrams (being valued in the $(2, 1)$-category of groupoids).
	
	However, \Cref{lcof-simp-inj-fin} tells that we can safely use the $(\infty, 1)$-categorical formulation, only that if the target $\infty$-category is an $(n, 1)$-category, we can discard the terms above level $n$ without changing the limits.
\end{remark}
\begin{remark}
\Cref{$n$-cofinal-lim-pres} (4) in particular tells that if a functor $p\colon\EuScript{C}\to\EuScript{D}$ is right $n$-cofinal and an $(n, 1)$-category $\EuScript{E}$ admits all $\EuScript{C}$-shaped colimits, it also admits all $\EuScript{D}$-shaped colimits (similarly, if a functor preserves $\EuScript{C}$-shaped colimits, then it also preserves $\EuScript{D}$-shaped colimits). However, the converse cannot hold. We provide a counterexample in the case $n=1$ with $p$ the inclusion functor $(\mathbf{\Delta}_{\rm s}^{\leqslant 1})^{\rm op}\hookrightarrow(\mathbf{\Delta}^{\leqslant 1})^{\rm op}$ here. We take $\EuScript{E}$ to be the category
\newcommand{\stackspace}{3.2}
\newcommand{\stack}[2][1cm]{\;\tikz[baseline, yshift=.65ex]%
	{\foreach \k [evaluate=\k as \r using (.5*#2+.5-\k)*\stackspace] in {1,...,#2}{%
			\ifodd\k{\draw[->](0,\r pt)--(#1,\r pt);}%
			\else{\draw[<-, densely dashed, blue](0,\r pt)--(#1,\r pt);}\fi
	}}\;}
\[
\begin{tikzcd}
	e	&	d	\ar[l, "q"']	& a \arrow[l, shift left, "v"]\arrow[l, shift right, "u"']	\stack{3}b\stack{1}c,
\end{tikzcd}\]
where the part from $a$ to the right is a reflexive coequaliser diagram, while $q\circ u\ne q\circ v$. Then $\EuScript{E}$ admits reflexive coequalisers, but has no coequaliser of $u, v$ --- no morphism in $\EuScript{E}$ could coequalise them. One can make this construction even more explicit by considering sets and maps.
\end{remark}

\section{$n$-siftedness}

In this section, we investigate the notion of $n$-siftedness, the case $n=1$ is treated in \cite{ARV1sift}. It is designed in a way that $n$-siftedness interpolates between the usual notion of siftedness for ordinary categories and siftedness for $\infty$-categories by Lurie.

\begin{definition}[$n$-siftedness]\label{$n$-siftedness}
	Let $K$ be a simplicial set, let $-2\leqslant n\leqslant\infty$. We say that $K$ is \emph{$n$-sifted} if $K\ne\varnothing$ and its diagonal $\delta\colon K\to K\times K$ is right $n$-cofinal; equivalently, the diagonal map $\delta\colon K\to K^I$ is right $n$-cofinal for every finite set $I$. We say that $K$ is \emph{$n$-cosifted} if $K^{\rm op}$ is $n$-sifted, i.e. $K\ne\varnothing$ and its diagonal $\delta\colon K\to K\times K$ is left $n$-cofinal.
\end{definition}
\begin{remark}
\begin{enumerate}[label=(\arabic*)]
	\item If two simplicial sets $K, L$ are $n$-sifted, then so is $K\times L$.
	\item 
	An $\infty$-category $\EuScript{C}$ is $n$-sifted if and only if it is non-empty and, for every pair of objects $a, b\in\EuScript{C}$, the $\infty$-category $\EuScript{C}_{a/}\times_{\EuScript{C}}\EuScript{C}_{b/}$ is $n$-connective. This is clear since $\EuScript{C}\times_{\EuScript{C}\times\EuScript{C}}(\EuScript{C}\times\EuScript{C})_{(a, b)/}\cong\EuScript{C}_{a/}\times_{\EuScript{C}}\EuScript{C}_{b/}$ in $\sset$.
	
	In the same spirit, we find that an $\infty$-category $\EuScript{C}$ is $n$-sifted if and only if it is non-empty and, for every $m$-tuple of objects $a_1, \cdots, a_m\in\EuScript{C}$ ($m\geqslant 0$), the ($1$-categorical) limit simplicial set of the partial $m$-cube in $\sset$ whose last $m$ edges are $\EuScript{C}_{a_i/}\to\EuScript{C}$, $i=1, \cdots, m$, is $n$-connective.
	\item 
	If an $\infty$-category $\EuScript{C}$ is non-empty and has coproduct of any pair of objects, it is ($\infty$-)sifted. In fact, in this case, the diagonal $\EuScript{C}\to\EuScript{C}\times\EuScript{C}$ is a right adjoint.
	\item If a simplicial set $K$ is $n$-sifted, then both projections $K\times K\rightrightarrows K$ are right $n$-cofinal. The assertion in the above definition for $I=\varnothing$ follows from this (with a Joyal fibrant model $K\to\EuScript{C}$ taken: by \cite[Corollary 2.11 (2)]{COJ}, for any object $c\in\EuScript{C}$, $|\EuScript{C}|\simeq|\EuScript{C}_{c/}\times\EuScript{C}|$ is $n$-connective), while for $I$ non-empty this can be seen by induction.
\end{enumerate}
\end{remark}
\begin{prop}\label{nsift-cof-emb}
Let $f\colon A\to B$ be a left $n$-cofinal map in $\sset$. If $A$ is $n$-cosifted, then so is $B$.

The converse holds if $f$ is also a full embedding of $\infty$-categories.
\end{prop}
\begin{proof}
We have the following commutative diagram in $\sset$:
\[\begin{tikzcd}[column sep=4em]
	A \ar[r, "\delta_A"] \ar[d, "f"']  &  A\times A \arrow[d, "g=f\times f"]\\
	B \ar[r, "\delta_B"']      &  B\times B.
\end{tikzcd}\]
Since the vertical arrows are left $n$-cofinal, the result follows from \Cref{ncof-comp}.
\end{proof}
\begin{prop}
A non-empty simplicial set $C$ is $n$-sifted if and only if the opposite of its $\infty$-category of simplices, $\operatorname{N}(\mathbf{\Delta}_{/C})^{\rm op}$, is $n$-sifted.
\end{prop}
\begin{proof}
Since the comparison map $\operatorname{N}(\mathbf{\Delta}_{/C})^{\rm op}\to C, (\Delta^n\xrightarrow{\sigma}C)\mapsto\sigma(0)$ is a colocalisation, we can apply the previous proposition by taking a Joyal fibrant replacement of $C$.
\end{proof}
An easy application of \Cref{$n$-cofinal-lim-pres} (3) also gives the following invariance property of $n$-siftedness.
\begin{prop}
Let $f\colon A\to B$ is a categorical equivalence in $\sset$. Then $A$ is $n$-cosifted if and only if $B$ is $n$-cosifted.
\end{prop}
\begin{prop}\label{sift-finprod}
A simplicial set $K$ is $n$-sifted if and only if the functor $\operatorname{colim}\colon\operatorname{Fun}(K, \tau_{\leqslant n-1}\EuScript{S})\to\tau_{\leqslant n-1}\EuScript{S}$ preserves finite products.
\end{prop}
\begin{proof}
By taking a Joyal fibrant model of $K$, we can assume $K$ is an $\infty$-category (as the desired statement is invariant under categorical equivalence in $\sset$). For any pair of functors $X, Y\colon K\rightrightarrows\tau_{\leqslant n-1}\EuScript{S}$, we have a comparison map $\operatorname{colim}_{a\in K}(Xa\times Ya)\to\operatorname{colim}_{a\in K}Xa\times\operatorname{colim}_{a\in K}Ya$. This comparison map is identified with $\operatorname{colim}\delta^*F\to\operatorname{colim}F$, where $F=X\boxtimes Y\colon K\times K\to\tau_{\leqslant n-1}\EuScript{S}, (a, b)\mapsto Xa\times Yb$, and $\delta\colon K\to K\times K$ is the diagonal map; the product of $X, Y$ in $\operatorname{Fun}(K, \tau_{\leqslant n-1}\EuScript{S})$ is given by $\delta^*F=\delta^*(X\boxtimes Y)$.

Assume that $K$ is $n$-sifted, i.e. the diagonal map $\delta\colon K\to K\times K$ is right $n$-cofinal. Then we have $\operatorname{colim}_{a\in K}(Xa\times Ya)\xrightarrow{\sim}\operatorname{colim}_{a\in K}Xa\times\operatorname{colim}_{a\in K}Ya$. The functor $\operatorname{colim}\colon\operatorname{Fun}(K, \tau_{\leqslant n-1}\EuScript{S})\to\tau_{\leqslant n-1}\EuScript{S}$ clearly preserves the empty product; note that colimits are universal in $\tau_{\leqslant n-1}\EuScript{S}$. So it preserves finite products.

Conversely, assume that the functor $\operatorname{colim}\colon\operatorname{Fun}(K, \tau_{\leqslant n-1}\EuScript{S})\to\tau_{\leqslant n-1}\EuScript{S}$ preserves finite products. If $K=\varnothing$, then the functor $\operatorname{colim}\colon\operatorname{Fun}(K, \tau_{\leqslant n-1}\EuScript{S})\to\tau_{\leqslant n-1}\EuScript{S}$ would not preserve the empty product: the comparison map becomes $\varnothing\to*$ in $\tau_{\leqslant n-1}\EuScript{S}$; so $K\ne\varnothing$. Taking $X=\tau_{\leqslant n-1}\operatorname{Map}_K(a, -), Y=\tau_{\leqslant n-1}\operatorname{Map}_K(b, -)$ for objects $a, b\in K$, then we have $F=X\boxtimes Y=\operatorname{Map}_{K\times K}((a, b), -)$. We see that $\operatorname{colim}\delta^*F\xrightarrow{\sim}\operatorname{colim}F$. So the diagonal map $\delta\colon K\to K\times K$ is right $n$-cofinal by \Cref{$n$-cofinal-lim-pres} (9).
\end{proof}
\begin{example}\label{sift-simp-inj-fin}
	\emph{For $n\in\mathbb{N}$, the category $\mathbf{\Delta}^{\leqslant n}$ is $n$-cosifted.}
	
	This follows from \Cref{simp-ncofinal} and \Cref{nsift-cof-emb} applying to the full embedding $\mathbf{\Delta}^{\leqslant n}\hookrightarrow\mathbf{\Delta}$.
	
	Note however that the category
	\[\EuScript{C}=(\mathbf{\Delta}^{\leqslant 1}_{\rm s})^{\rm op}: [1]\rightrightarrows[0]\]
	is \emph{not} $1$-sifted. In fact, it is easy to see that $\EuScript{C}_{[0]/}\times_{\EuScript{C}}\EuScript{C}_{[1]/}$ is not connected (i.e., not $1$-connective).
\end{example}
Recall that any $\infty$-category $\EuScript{C}$ has a free completion under finite coproducts (see \cite[\S 3.3]{Rez22}), which we denote by $\operatorname{Fam}\EuScript{C}$ following \cite{ARV1sift}; in the notation of \cite{Rez22}, it is $\operatorname{PSh}^{\set^{<\omega}}(\EuScript{C})$. It is the full subcategory of $\EuScript{P}(\EuScript{C})$ consisting of finite coproducts of representable presheaves and has the following universal property: for any $\infty$-category $\EuScript{E}$ which has finite coproducts, restriction along the Yoneda embedding $\EuScript{C}\hookrightarrow\operatorname{Fam}\EuScript{C}\subset\EuScript{P}(\EuScript{C})$ induces an equivalence
\[\operatorname{Fun}^{\rm \coprod}(\operatorname{Fam}\EuScript{C}, \EuScript{E})\xrightarrow{\sim}\operatorname{Fun}(\EuScript{C}, \EuScript{E}).\]
\begin{prop}\label{sift-yoneda}
An $\infty$-category $\EuScript{C}$ is $n$-sifted if and only if the Yoneda functor $h\colon\EuScript{C}\hookrightarrow\operatorname{Fam}\EuScript{C}$ is right $n$-cofinal.
\end{prop}
\begin{proof}
Since $\operatorname{Fam}\EuScript{C}$ has finite coproducts, it is ($\infty$-)sifted. So if the Yoneda functor $h\colon\EuScript{C}\hookrightarrow\operatorname{Fam}\EuScript{C}$ is right $n$-cofinal, $\EuScript{C}$ is $n$-sifted.

Conversely, if $\EuScript{C}$ is $n$-sifted, then for every $m$-tuple of objects $a_1, \cdots, a_m\in\EuScript{C}$ ($m\geqslant 0$), the ($1$-categorical) limit of the partial $m$-cube in $\sset$ whose last $m$ edges are $\EuScript{C}_{a_i/}\to\EuScript{C}$, $i=1, \cdots, m$, is $n$-connective. This exactly says that, for $X=\coprod_{1\leqslant i\leqslant m} h_{a_i}\in\operatorname{Fam}\EuScript{C}$, the simplicial set $\EuScript{C}\times_{\operatorname{Fam}\EuScript{C}}(\operatorname{Fam}\EuScript{C})_{X/}$ is $n$-connective (note that $\EuScript{C}\times_{\operatorname{Fam}\EuScript{C}}(\operatorname{Fam}\EuScript{C})_{h_{a_i}/}$ is $\EuScript{C}_{a_i/}$). Hence $h\colon\EuScript{C}\hookrightarrow\operatorname{Fam}\EuScript{C}$ is right $n$-cofinal.
\end{proof}
\begin{prop}\label{n-siftcol-pres}
	Let $-1\leqslant n\leqslant\infty$, let $\EuScript{C, D}$ be $(n, 1)$-categories with $\EuScript{C}$ admits finite coproducts. Then a functor $F\colon\EuScript{C}\to\EuScript{D}$ preserves all $n$-sifted colimits which exist in $\EuScript{C}$ if and only if $F$ preserves all $\infty$-sifted colimits which exist in $\EuScript{C}$.
\end{prop}
\begin{proof}
	Let $C$ be a small $n$-sifted simplicial set and let $v\colon C\to\EuScript{C}$ be a diagram admitting a colimit $a\in\EuScript{C}$.
	
	We take a Joyal fibrant model $C\to\EuScript{I}$ that is inner anodyne in $\sset$, and we can extend $v\colon C\to\EuScript{C}$ to a functor $\EuScript{I}\to\EuScript{C}$ (since $\EuScript{C}^{\EuScript{I}}\to\EuScript{C}^C$ is a trivial fibration, see \cite[Corollary 2.3.2.5]{HTT}), which in turn is equivalent to a functor $v'\colon\EuScript{I}'=\operatorname{Fam}\EuScript{I}\to\EuScript{C}$ that preserves finite coproducts (\cite[\S 3.3]{Rez22}). Since $\EuScript{C, D}$ are $(n, 1)$-categories and $C\to\EuScript{I}\hookrightarrow\EuScript{I}'=\operatorname{Fam}\EuScript{I}$ is right $n$-cofinal, $v, v'$ have the same colimit $a\in\EuScript{C}$; as do $F\circ v$ and $F\circ v'$. Moreover, $\EuScript{I}'=\operatorname{Fam}\EuScript{I}$ is $\infty$-sifted, so we have
	\[\operatorname{colim}F\circ v\simeq \operatorname{colim}F\circ v'\simeq F(\operatorname{colim}v')\simeq F(a),\]
	as desired.
\end{proof}
\begin{prop}
Let $-1\leqslant n\leqslant\infty$, let $\EuScript{C}$ be an $(n, 1)$-category admitting finite coproducts. Then $\EuScript{C}$ admits all $n$-sifted colimits if and only if $\EuScript{C}$ admits all filtered colimits and $(\mathbf{\Delta}^{\leqslant n})^{\rm op}$-shaped colimits, if and only if $\EuScript{C}$ admits all small colimits.
\end{prop}
\begin{proof}
	The forward direction is clear, since filtered and $(\mathbf{\Delta}^{\leqslant n})^{\rm op}$-shaped diagrams are $n$-sifted. We only need to treat the backward direction. So let $\EuScript{C}$ have finite coproducts, filtered colimits and $(\mathbf{\Delta}^{\leqslant n})^{\rm op}$-shaped colimits, then it has all coproducts. The result now follows from \Cref{colimits-simplicial}.
\end{proof}
\begin{remark}
This existence result cannot hold in general without assuming $\EuScript{C}$ has finite coproducts: \cite[\S 1.4]{ARV1sift} provides a counterexample in the case $n=1$.
\end{remark}
The following result is an analogy of \cite[Corollary 5.5.8.17]{HTT} (where it should also assume that $\EuScript{D}$ admits filtered colimits and $\mathbf{\Delta}^{\rm op}$-shaped colimits).
\begin{theorem}\label{nsift-colim-pres}
Let $-1\leqslant n\leqslant\infty$, let $\EuScript{C, D}$ be $(n, 1)$-categories with $\EuScript{C}$ cocomplete and $\EuScript{D}$ admits filtered colimits and $(\mathbf{\Delta}^{\leqslant n})^{\rm op}$-shaped colimits. Then a functor $F\colon\EuScript{C}\to\EuScript{D}$ preserves $n$-sifted colimits if and only if $F$ preserves filtered colimits and $(\mathbf{\Delta}^{\leqslant n})^{\rm op}$-shaped colimits.
\end{theorem}
\begin{proof}
	The proof is as that of \cite[Corollary 5.5.8.17]{HTT}; we repeat that argument for convenience.
	
	Again, we only need to treat the backward direction. So assume that $F$ preserves filtered colimits and $(\mathbf{\Delta}^{\leqslant n})^{\rm op}$-shaped colimits. Since $\EuScript{C, D}$ are $(n, 1)$-categories, $F$ preserves filtered colimits and $\mathbf{\Delta}^{\rm op}$-shaped colimits (\Cref{$n$-cofinal-lim-pres} and \Cref{simp-ncofinal}). By \Cref{n-siftcol-pres}, we only need to show that $F$ preserves $\infty$-sifted colimits.
	
	Let $C$ be a small $\infty$-sifted simplicial set and let  $v\colon C\to\EuScript{C}$ be a diagram admitting a colimit $a\in\EuScript{C}$. As in proof of \Cref{n-siftcol-pres}, we take a right cofinal map $i\colon C\to\EuScript{I}'$ in $\sset$ with $\EuScript{I}'$ an $\infty$-category admitting finite coproducts, and $v\simeq v'\circ i$ for some $v'\in\operatorname{Fun}^{\coprod}(\EuScript{I}', \EuScript{C})$. By \cite[Proposition 5.5.8.15]{HTT} (or \Cref{yoneda-decomfun} below), $v'\simeq q\circ h$ for some $q\in\operatorname{Fun}^{\rm colim}(\EuScript{P}_{\Sigma}(\EuScript{I}'), \EuScript{C})$, where $h\colon\EuScript{I}'\hookrightarrow\EuScript{P}_{\Sigma}(\EuScript{I}')$ is the Yoneda embedding, as shown in the following diagram:
	\[\begin{tikzcd}[column sep=4em]
		C \ar[r, "v"] \ar[d, "i"']  &  \EuScript{C} \arrow[r, "F"]	  &	\EuScript{D}\\
		\EuScript{I}' \ar[r, "h", near end]   \ar[ru, "v'"]     &  \EuScript{P}_{\Sigma}(\EuScript{I}').\arrow[u, "q"']		&	
	\end{tikzcd}\]
	Let $\operatorname{colim}(h\circ i)=b\in\EuScript{P}_{\Sigma}(\EuScript{I}')$, then $a\simeq q(b)$.
	
	Since $F\circ v\simeq F\circ v'\circ i\simeq(F\circ q)\circ (h\circ i)$, and by \cite[Proposition 5.5.8.15 (2)]{HTT}, $F\circ q$ preserves sifted colimits. We thus obtain
	\[\operatorname{colim}F\circ v\simeq\operatorname{colim}(F\circ q)\circ(h\circ i)\simeq(F\circ q)(\operatorname{colim}h\circ i)\simeq F(q(b))\simeq F(a)\simeq F(\operatorname{colim}v).\]
	This is what we wanted to prove.
\end{proof}

\section{Free colimit completion and finiteness conditions}

We turn to some commonly used finiteness conditions in the $\infty$-categorical setting in this section.

Let $\mathcal{F}\subset\operatorname{Cat}_{\infty}$ be a class of small $\infty$-categories, and $\EuScript{E}$ an $\infty$-category which has $\mathcal{F}$-colimits. By \cite[\S 5]{Rez22}, $\mathcal{F}$ has a \emph{filtering closure}. An object $e\in\EuScript{E}$ is \emph{$\mathcal{F}$-compact} if the functor $\operatorname{Map}_{\EuScript{E}}(e, -)\colon\EuScript{E}\to\EuScript{S}$ preserves $\mathcal{F}$-colimits (\cite[\S 9]{Rez22}). We write $\EuScript{E}^{\mathcal{F}\operatorname{-cpt}}\subset\EuScript{E}$ for the full subcategory of $\mathcal{F}$-compact objects. By \Cref{colretr} below, $\EuScript{E}^{\mathcal{F}\operatorname{-cpt}}$ is closed  under retracts in $\EuScript{E}$.
\begin{prop}
	 Let $\mathcal{F}\subset\operatorname{Cat}_{\infty}$ be a class of small $\infty$-categories, and $\EuScript{E}$ an $\infty$-category. Assume that $\EuScript{E}$ has $\mathcal{F}$-colimits and the $\infty$-category $\EuScript{E}^{\mathcal{F}\operatorname{-cpt}}$ is essentially small. Then the canonical functor $\operatorname{PSh}^{\mathcal{F}}(\EuScript{E}^{\mathcal{F}\operatorname{-cpt}})\to\EuScript{E}$ is fully faithful, which is an equivalence if and only if $\EuScript{E}$ is generated under $\mathcal{F}$-colimits by $\EuScript{E}^{\mathcal{F}\operatorname{-cpt}}$.
\end{prop}
\begin{proof}
	 The canonical functor $\operatorname{PSh}^{\mathcal{F}}(\EuScript{E}^{\mathcal{F}\operatorname{-cpt}})\to\EuScript{E}$ is identity on $\EuScript{E}^{\mathcal{F}\operatorname{-cpt}}$ and in general is given by taking $\overline{\mathcal{F}}$-colimits in $\EuScript{E}$ \myparen{\cite[\S 3.3]{Rez22}}. The rest statement follows from \cite[\S 9.2]{Rez22}.
\end{proof}
\begin{lemma}[Colimit diagrams are stable under retracts]\label{colretr}
	Let $\EuScript{E}$ be an $\infty$-category and let $u\colon I^{\triangleright}\to\EuScript{E}$ be a colimit diagram. If $v\colon I^{\triangleright}\to\EuScript{E}$ is a diagram for which there exist edges $\alpha\colon u\to v, \beta\colon v\to u$ in $\operatorname{Fun}(I^{\triangleright}, \EuScript{E})$ with $\alpha\circ\beta\simeq\operatorname{id}_v$. Then $v$ is also a colimit diagram.
\end{lemma}
\begin{proof}
	We denote the cone objects of $u, v$ by $a, b\in\EuScript{E}$. Since colimits in $\EuScript{E}$, exist or not, can be tested on applying $\operatorname{Map}_{\EuScript{E}}(-, e)$ for objects $e\in\EuScript{E}$, we only need to show that the left vertical arrow in the following commutative diagram
	\[\begin{tikzcd}[column sep=3em]
		\operatorname{Map}_{\EuScript{E}}(b, e) \ar[r, "\alpha^*"] \ar[d]  &  \operatorname{Map}_{\EuScript{E}}(a, e) \ar[r, "\beta^*"] \ar[d]  &  \operatorname{Map}_{\EuScript{E}}(b, e)    \arrow[d]\\
		\operatorname{lim}\operatorname{Map}_{\EuScript{E}}(-, e)\circ v|_I   \ar[r, "\alpha^*"]   &  \operatorname{lim}\operatorname{Map}_{\EuScript{E}}(-, e)\circ u|_I   \ar[r, "\beta^*"]  &  \operatorname{lim}\operatorname{Map}_{\EuScript{E}}(-, e)\circ v|_I
	\end{tikzcd}\]
	in $\EuScript{S}$ is an isomorphism for every $e\in\EuScript{E}$. But this is clear, since it exhibits the left vertical arrow as a retract of the middle vertical arrow, the latter is an isomorphism in $\EuScript{S}$ by assumption, hence so is the former.
\end{proof}
\begin{lemma}\label{colim-in-cat-pts}
	Let $I$ be a small simplicial set, let $\EuScript{C}$ be a small $\infty$-category admitting all $I$-shaped colimits.
	\begin{enumerate}[label=\myparen{\arabic*}]
		\item If $X\in\operatorname{Fun}^{I^{\rm op}-\operatorname{lim}}(\EuScript{C}^{\rm op}, \EuScript{S})\subset\EuScript{P}(\EuScript{C})$ is an $I^{\rm op}$-limit preserving functor, then the $\infty$-category $\EuScript{C}_{/X}$ admits all $I$-shaped colimits and the projection $p\colon\EuScript{C}_{/X}\to\EuScript{C}$ preserves such colimits..
		\item Let $\EuScript{D}$ be another $\infty$-category and let $F\in\operatorname{Fun}^{I-\operatorname{colim}}(\EuScript{C}, \EuScript{D})$ be an $I$-colimit preserving functor, then for every object $d\in\EuScript{D}$, the $\infty$-category $\EuScript{C}_{/d}=\EuScript{C}\times_{\EuScript{D}}\EuScript{D}_{/d}$ admits all $I$-shaped colimits and the projection $q\colon\EuScript{C}_{/d}\to\EuScript{C}$ preserves such colimits.
	\end{enumerate}
\end{lemma}
\begin{proof}
	\begin{enumerate}[label=(\arabic*)]
		\item Note first that the functor $X\colon\EuScript{C}^{\rm op}\to\EuScript{S}$ classifies the right fibration $\EuScript{C}_{/X}\to\EuScript{C}$. Let $v\colon I\to\EuScript{C}_{/X}$ be a diagram and let $\bar{u}\colon I^{\triangleright}\to\EuScript{C}$ be a colimit diagram extending $u=p\circ v\colon I\to\EuScript{C}, i\mapsto a_i$, with cone object $a\in\EuScript{C}$ so that $v(i)\in X(a_i)$.
		
		By assumption, $X(a)\xrightarrow{\sim}\operatorname{lim}_{i\in I^{\rm op}}X(a_i)\in\EuScript{S}$, so these $v(i)\in X(a_i)$ specify an object $v(\infty)\in X(a)$; together with $v\colon I\to\EuScript{C}_{/X}$, it defines a diagram $\bar{v}\colon I^{\triangleright}\to\EuScript{C}_{/X}$. We need to prove that this is a colimit diagram in $\EuScript{C}_{/X}$, for which we only need to prove that by mapping it to an arbitrary object in $\EuScript{C}_{/X}$ yields a limit diagram.
		
		This is the case since, for $\alpha\in X(c), \alpha'\in X(c')$, the mapping space $\operatorname{Map}_{\EuScript{C}_{/X}}(\alpha, \alpha')$ is the fibre of $\alpha'_*\colon\operatorname{Map}_{\EuScript{C}}(c, c')\to X(c)$ over the object $\alpha$, and forming fibres commutes with limits.
		\item Since the functor $h_d\colon\EuScript{D}^{\rm op}\to\EuScript{S}$ classifies the right fibration $\EuScript{D}_{/d}\to\EuScript{D}$ and the unstraightening equivalences are compatible with base change, the functor $h_d\circ F^{\rm op}\colon\EuScript{C}^{\rm op}\to\EuScript{S}, c\mapsto\operatorname{Map}_{\EuScript{D}}(F(c), d)$ classifies the right fibration $q\colon\EuScript{C}_{/d}\to\EuScript{C}$. But by design, unstraightening of a functor $X\colon\EuScript{C}^{\rm op}\to\EuScript{S}$ is the projection $\EuScript{C}_{/X}\to\EuScript{C}$, so  $\EuScript{C}_{/X}\simeq\EuScript{C}_{/d}$. And by assumption, the functor $X:=h_d\circ F^{\rm op}$ preserves $I$-shaped limits. Thus (2) follows from (1).
	\end{enumerate}
\end{proof}
For a class $\mathcal{U}\subset\operatorname{Cat}_{\infty}$ of small $\infty$-categories, we write $\mathcal{U}^{\rm op}=\{U^{\rm op}\colon U\in\mathcal{U}\}$ and define $\operatorname{Filt}_{\mathcal{U}}\subset\operatorname{Cat}_{\infty}$ to be the class of all small $\infty$-categories $J$ for which the functor $\operatorname{colim}\colon\operatorname{Fun}(J, \EuScript{S})\to\EuScript{S}$ preserves $\mathcal{U}^{\rm op}$-limits in $\EuScript{S}$. It is a filtering class (\cite[\S 10.5]{Rez22}).
\begin{prop}
	Let $\mathcal{U}\subset\operatorname{Cat}_{\infty}$ be a class of small $\infty$-categories, let $\mathcal{F}:= \operatorname{Filt}_{\mathcal{U}}$.
	\begin{enumerate}[label=\myparen{\arabic*}]
		\item For an $\infty$-category $\EuScript{S}$ admitting $\mathcal{F}$-colimits, its subcategory $\EuScript{E}^{\mathcal{F}\operatorname{-cpt}}$ of $\mathcal{F}$-compact objects is closed under retracts and under $\mathcal{U}$-colimits which exist in $\EuScript{E}$.
		\item Let $\EuScript{C}$ be an $\infty$-category. If $X\in\EuScript{P}(\EuScript{C})$ is a retract of the colimit of a diagram of the form $I\to\EuScript{C}\xrightarrow{h}\EuScript{P}(\EuScript{C})$ for some $I\in\mathcal{U}$, then $X\in\EuScript{P}(\EuScript{C})^{\mathcal{F}\operatorname{-cpt}}$.
	\end{enumerate}
\end{prop}
\begin{proof}
\begin{enumerate}[label=(\arabic*)]
	\item Since by definition, $\mathcal{F}$-colimits commute with $\mathcal{U}^{\rm op}$-limits in $\EuScript{S}$, we easily find that $\EuScript{E}^{\mathcal{F}\operatorname{-cpt}}$ is closed under $\mathcal{U}$-colimits which exist in $\EuScript{E}$.
	\item Just note that objects in the image of the Yoneda embedding are  \emph{completely compact} (\cite[Definition 5.1.6.2]{HTT}), or \emph{atomic} (\cite[Definition 2.4]{COJ}) and by (1), $\EuScript{P}(\EuScript{C})^{\mathcal{F}\operatorname{-cpt}}$ is closed under retracts and under $\mathcal{U}$-colimits in $\EuScript{P}(\EuScript{C})$.
\end{enumerate}
\end{proof}
\begin{remark}
	It might be interesting to see if the converse of (2) holds. One possible way is to show that for $X\in\EuScript{P}(\EuScript{C})^{\mathcal{F}\operatorname{-cpt}}$, there exists a right cofinal functor $I\to\EuScript{C}_{/X}$ with $I\in\mathcal{U}$.
\end{remark}
\begin{theorem}\label{colim-yon}
	Let $\mathcal{U}\subset\operatorname{Cat}_{\infty}$ be a class of small $\infty$-categories, let $\mathcal{F}:= \operatorname{Filt}_{\mathcal{U}}$.
	\begin{enumerate}[label=\myparen{\arabic*}]
		\item For a small $\infty$-category $\EuScript{C}$, any object $X\in\operatorname{PSh}^{\mathcal{F}}(\EuScript{C})$, viewed as a functor $\EuScript{C}^{\rm op}\to\EuScript{S}$, preserves $\mathcal{U}^{\rm op}$-limits which exist in $\EuScript{C}^{\rm op}$.
		\item For a small $\infty$-category $\EuScript{C}$, the Yoneda functor $h\colon\EuScript{C}\hookrightarrow\operatorname{PSh}^{\mathcal{F}}(\EuScript{C})$ preserves $\mathcal{U}$-colimits which exist in $\EuScript{C}$.
		\item Assume that any small $\infty$-category admitting $\mathcal{U}$-colimits lies in $\mathcal{F}$, then for any small $\infty$-category $\EuScript{C}$ admitting $\mathcal{U}$-colimits, we have
		\[\operatorname{PSh}^{\mathcal{F}}(\EuScript{C})=\operatorname{Fun}^{\rm \mathcal{U}^{\rm op}-lim}(\EuScript{C}^{\rm op}, \EuScript{S})\subset\EuScript{P}(\EuScript{C}).\]
		Moreover, in this case, the $\infty$-category  $\operatorname{PSh}^{\mathcal{F}}(\EuScript{C})$ is closed under limits in $\EuScript{P}(\EuScript{C})$.
		\item 
		\newcommand*\cocolon{%
			\nobreak
			\mskip6mu plus1mu
			\mathpunct{}%
			\nonscript
			\mkern-\thinmuskip
			{:}
			\mskip2mu
			\relax
		}
		Let $\EuScript{C}$ be a small $\infty$-category admitting $\mathcal{U}$-colimits for which $\operatorname{PSh}^{\mathcal{F}}(\EuScript{C})=\operatorname{Fun}^{\rm \mathcal{U}^{\rm op}-lim}(\EuScript{C}^{\rm op}, \EuScript{S})\subset\EuScript{P}(\EuScript{C})$ and let $F\colon\EuScript{P}(\EuScript{C})\rightleftarrows\EuScript{D}\cocolon G$ be an adjunction. Then $G$ factors through $\operatorname{PSh}^{\mathcal{F}}(\EuScript{C})$ if and only if $F\circ h\in\operatorname{Fun}^{\rm \mathcal{U}-colim}(\EuScript{C}, \EuScript{D})$, where $h\colon\EuScript{C}\hookrightarrow\EuScript{P}(\EuScript{C})$ is the Yoneda functor.
		
		If this is the case, we obtain a restricted adjunction
		\[F\colon\operatorname{PSh}^{\mathcal{F}}(\EuScript{C})\rightleftarrows\EuScript{D}\cocolon G.\]
	\end{enumerate}
\end{theorem}
\begin{proof}
	\begin{enumerate}[label=(\arabic*)]
		\item Let $u\colon(I^{\rm op})^{\triangleleft}\to\EuScript{C}^{\rm op}$ be a limit diagram with $I\in\mathcal{U}$. By writing $X$ as a colimit of representables over an $\infty$-category in $\mathcal{F}$ and using that $\mathcal{F}$-colimits commute with $\mathcal{U}^{\rm op}$-limits in $\EuScript{S}$, we easily find that $X\circ u\colon(I^{\rm op})^{\triangleleft}\to\EuScript{S}$ is also a limit diagram.
		\item Let $v\colon I^{\triangleright}\to\EuScript{C}$ be a colimit diagram with $I\in\mathcal{U}$. By the Yoneda lemma, we only need to show that, for any $X\in\operatorname{PSh}^{\mathcal{F}}(\EuScript{C})$, the functor $L\circ(h\circ v)^{\rm op}\colon(I^{\rm op})^{\triangleleft}\to\widehat{\EuScript{S}}$ is a limit diagram, where $L=\operatorname{Map}_{\EuScript{P}(\EuScript{C})}(-, X)\colon\operatorname{PSh}^{\mathcal{F}}(\EuScript{C})^{\rm op}\to\widehat{\EuScript{S}}$ is the functor represented by $X$. Again by the Yoneda lemma, we have $L\circ h^{\rm op}\simeq X\in\operatorname{PSh}^{\mathcal{F}}(\EuScript{C})$. The result now follows from (1).
		\item By (1), $\operatorname{PSh}^{\mathcal{F}}(\EuScript{C})\subset\operatorname{Fun}^{\rm \mathcal{U}^{\rm op}-lim}(\EuScript{C}^{\rm op}, \EuScript{S})$. Conversely, for any $X\in\operatorname{Fun}^{\rm \mathcal{U}^{\rm op}-lim}(\EuScript{C}^{\rm op}, \EuScript{S})$, the $\infty$-category $\EuScript{C}_{/X}$ admits all $\mathcal{U}$-colimits by the previous proposition, so $\EuScript{C}_{/X}\in\mathcal{F}$ by assumption, and hence $X\in\operatorname{PSh}^{\mathcal{F}}(\EuScript{C})$ (\cite[\S 4.2]{Rez22}).
		The last statement holds as taking limits commutes with each other.
		\item Since colimits in $\EuScript{D}$, exist or not, can be tested on applying $h'_d=\operatorname{Map}_{\EuScript{D}}(-, d)$ for objects $d\in\EuScript{D}$, $F\circ h\in\operatorname{Fun}^{\rm \mathcal{U}-colim}(\EuScript{C}, \EuScript{D})$ if and only if, for every object $d\in\EuScript{D}$, we have $h'_d\circ(F\circ h)^{\rm op}\in\operatorname{Fun}^{\rm \mathcal{U}^{\rm op}-lim}(\EuScript{C}^{\rm op}, \EuScript{S})=\operatorname{PSh}^{\mathcal{F}}(\EuScript{C})$. This is the case if and only if $G$ factors through $\operatorname{PSh}^{\mathcal{F}}(\EuScript{C})$, since by adjunction and the Yoneda lemma, $h'_d\circ(F\circ h)^{\rm op}\simeq G(d)\in\EuScript{P}(\EuScript{C})$.
		\end{enumerate}
\end{proof}
\begin{remark}
	\begin{enumerate}[label=(\arabic*)]
		\item Though the inclusion $\operatorname{PSh}^{\mathcal{F}}(\EuScript{C})\hookrightarrow\EuScript{P}(\EuScript{C})$ preserves all $\mathcal{F}$-colimits, it does not preserve $\mathcal{U}$-colimits in general: for $I\in\mathcal{U}$ and a diagram $v\colon I\to\EuScript{C}$, the colimits of $I\xrightarrow{v}\EuScript{C}\xhookrightarrow{h}\operatorname{PSh}^{\mathcal{F}}(\EuScript{C})$ and of $I\xrightarrow{v}\EuScript{C}\xhookrightarrow{h}\EuScript{P}(\EuScript{C})$ are not equivalent in $\EuScript{P}(\EuScript{C})$ in general.
		\item The assumption in statement (3) is satisfied in the situations described in the example below (and with $\mathcal{U}^{\rm op}=\mathcal{U}$). There should be certain conditions on the class $\mathcal{U}$, under which the assumption in statement (3) is automatically fulfilled. I will not explore such conditions here.
	\end{enumerate}
\end{remark}
\begin{example}
	Here are some important special cases of the previous discussion in the $\infty$-categorical setting, which are commonly used in the $1$-categorical setting. We fix an $\infty$-category $\EuScript{C}$.
	\begin{enumerate}[label=(\arabic*)]
		\item Take $\mathcal{F}=\operatorname{Filt}_{\operatorname{Sm}_{\kappa}}$ to be the collection of all $\kappa$-filtered $\infty$-categories for a regular cardinal $\kappa$, then $\operatorname{PSh}^{\mathcal{F}}(\EuScript{C})=\operatorname{Ind}_{\kappa}(\EuScript{C})$ (\cite[\S 10.7]{Rez22}).
		
		If $\EuScript{C}$ admits all $\kappa$-small colimits, then ${\rm Ind}_{\kappa}(\EuScript{C})=\operatorname{Fun}^{\rm \kappa-lim}(\EuScript{C}^{\rm op}, \EuScript{S})\subset\EuScript{P}(\EuScript{C})$ consists of those functors preserve $\kappa$-small limits. If $\EuScript{C}$ admits all finite colimits, then ${\rm Ind}(\EuScript{C})=\operatorname{Fun}^{\rm lex}(\EuScript{C}^{\rm op}, \EuScript{S})\subset\EuScript{P}(\EuScript{C})$ consists of all left exact functors \myparen{those preserve finite limits}.
		
		For an $\infty$-category $\EuScript{E}$ having $\kappa$-filtered colimits, we denote $\EuScript{E}^{\mathcal{F}\operatorname{-cpt}}=\EuScript{E}^{\kappa}$, and call it the full subcategory of \emph{$\kappa$-compact} objects in $\EuScript{E}$. It is closed under $\kappa$-small colimits which exist in $\EuScript{E}$ (since $\kappa$-filtered colimits commute with $\kappa$-small limits in $\EuScript{S}$). The inclusion functor $\EuScript{E}^{\kappa}\hookrightarrow\EuScript{E}$ corresponds to a fully faithful $\kappa$-filtered colimit preserving functor $\operatorname{Ind}_{\kappa}(\EuScript{E}^{\kappa})\to\EuScript{E}$ which is identity on $\EuScript{E}^{\kappa}$ and in general is given by taking $\kappa$-filtered colimits in $\EuScript{E}$.
		
		In the case $\kappa=\omega$, for an $\infty$-category $\EuScript{E}$ which has ($\omega$-)filtered colimits, we have $\EuScript{E}^{\mathcal{F}\operatorname{-cpt}}=\EuScript{E}^{\rm fp}=\EuScript{E}^{\omega}$. It is closed under finite colimits which exist in $\EuScript{E}$; its objects are called \emph{compact} or \emph{of finite presentation}. The inclusion functor $\EuScript{E}^{\rm fp}\hookrightarrow\EuScript{E}$ corresponds to a fully faithful filtered colimit preserving functor $\operatorname{Ind}(\EuScript{E}^{\rm fp})\to\EuScript{E}$ which is identity on $\EuScript{E}^{\rm fp}$ and in general is given by taking filtered colimits in $\EuScript{E}$.
		
		\vspace{2mm}
		
		If we only assume that $\EuScript{E}$ has filtered colimits of monomorphisms (i.e. all transition morphisms are monomorphisms), we say that an object $e\in\EuScript{E}$ is \emph{finitely generated} or \emph{of finite type} if the functor $\operatorname{Map}_{\EuScript{E}}(e, -)\colon\EuScript{E}\to\EuScript{S}$ commutes with filtered colimits of monomorphisms. The full subcategory of finitely generated objects in $\EuScript{E}$ is denoted by $\EuScript{E}^{\rm fg}$. It is closed under finite colimits which exist in $\EuScript{E}$.
		
		\vspace{2mm}
		
		\item Take $\mathcal{F}=\operatorname{Filt}_{\operatorname{Set}^{<\omega}}$ to be the collection of all small sifted $\infty$-categories, then $\operatorname{PSh}^{\mathcal{F}}(\EuScript{C})=\operatorname{sInd}(\EuScript{C})$, the \emph{sifted completion} of $\EuScript{C}$ (\cite[\S 10.8]{Rez22}). If $\EuScript{C}$ admits all finite coproducts, then $\operatorname{sInd}(\EuScript{C})=\EuScript{P}_{\Sigma}(\EuScript{C})=\operatorname{Fun}^{\times}(\EuScript{C}^{\rm op}, \EuScript{S})\subset\EuScript{P}(\EuScript{C})$ (see \Cref{colim-yon} (3) or \cite[Propositions 5.5.8.10 and 5.5.8.15]{HTT}).
		
		For an $\infty$-category $\EuScript{E}$ having small sifted colimits, we denote $\EuScript{E}^{\mathcal{F}\operatorname{-cpt}}=\EuScript{E}^{\rm sfp}$, it is closed under finite coproducts which exist in $\EuScript{E}$ (since sifted colimits commute with finite products in $\EuScript{S}$); its objects are called \emph{strongly of finite presentation} (or \emph{compact projective}, when $\EuScript{E}$ has all colimits). The inclusion functor $\EuScript{E}^{\rm sfp}\hookrightarrow\EuScript{E}$ corresponds to a fully faithful sifted colimit preserving functor $\operatorname{sInd}(\EuScript{E}^{\rm sfp})\to\EuScript{E}$ which is identity on $\EuScript{E}^{\rm sfp}$ and in general is given by taking sifted colimits in $\EuScript{E}$. If $\EuScript{E}$ is an $(n, 1)$-category, then for any $a\in\EuScript{E}^{\rm sfp}$, the functor $\operatorname{Map}_{\EuScript{E}}(a, -)=\tau_{\leqslant n-1}\operatorname{Map}_{\EuScript{E}}(a, -)\colon\EuScript{E}\to\tau_{\leqslant n-1}\EuScript{S}$ preserves all $n$-sifted colimits (by \Cref{n-siftcol-pres}).
	\end{enumerate}
\end{example}
\begin{prop}
	Let $\EuScript{C}$ be a small $\infty$-category and let $X\in\operatorname{sInd}(\EuScript{C})$, then the $\infty$-category $\EuScript{C}_{/X}$ is sifted.
\end{prop}
Again, this follows from \cite[\S 4.2 and 10.8]{Rez22}; in fact, $\EuScript{C}_{/X}$ admits all finite coproducts if $\EuScript{C}$ does (by \Cref{colim-in-cat-pts}).
\begin{prop}
	Let $\EuScript{C}$ be a small $\infty$-category, let $\kappa$ be a regular cardinal. Then the inclusions
	\[\EuScript{C}\hookrightarrow\operatorname{Ind}_{\kappa}(\EuScript{C})\hookrightarrow\operatorname{Ind}(\EuScript{C})\hookrightarrow\operatorname{sInd}(\EuScript{C})\]
	are left cofinal.
\end{prop}
\begin{proof}
	Let $X\in\operatorname{sInd}(\EuScript{C})$, then $\EuScript{C}_{/X}$ is sifted, hence  is weakly contractible. So $\EuScript{C}\hookrightarrow\operatorname{sInd}(\EuScript{C})$ is left cofinal. Now apply \Cref{ncof-comp} (2).
\end{proof}
	The same argument shows that the inclusion $\EuScript{C}\hookrightarrow\operatorname{PSh}^{\mathcal{F}}(\EuScript{C})$ is left cofinal, where $\mathcal{F}= \operatorname{Filt}_{\{\varnothing\}}$ is the class of weakly contractible $\infty$-categories (\cite[\S 10.10]{Rez22}). In fact, we have the following more general result.
\begin{prop}
Let $\mathcal{F}\subset\operatorname{Cat}_{\infty}$ be the class of small \emph{(weakly)} $n$-connective $\infty$-categories.
\begin{enumerate}[label=\myparen{\arabic*}]
	\item The class $\mathcal{F}$ is a filtering class.
	\item For any small $\infty$-category $\EuScript{C}$, the Yoneda embedding $h\colon\EuScript{C}\hookrightarrow\operatorname{PSh}^{\mathcal{F}}(\EuScript{C})$ is left $n$-cofinal, and $\operatorname{PSh}^{\mathcal{F}}(\EuScript{C})$ is the \emph{largest} full subcategory of $\EuScript{P}(\EuScript{C})$ which contains $\EuScript{C}$ as a left $n$-cofinal subcategory \emph{(and any other full subcategory in between also does)}.
\end{enumerate}
\end{prop}
\begin{proof}
	\begin{enumerate}[label=(\arabic*)]
		\item By definition, we have to show that if $\operatorname{PSh}^{\mathcal{F}}(\EuScript{C})$ contains the terminal presheaf $\mathbf{1}=\underline{*}$, then $\EuScript{C}$ is weakly $n$-connective. Indeed, by \Cref{cofintopt-cst}, $\mathbf{1}=\underline{*}$ viewed as a functor $\operatorname{PSh}^{\mathcal{F}}(\EuScript{C})\xrightarrow{*}\tau_{\leqslant n-1}\EuScript{S}$, preserves $\mathcal{F}$-colimits. So it is a left Kan extension of the functor $\EuScript{C}\xrightarrow{*}\tau_{\leqslant n-1}\EuScript{S}$ along the Yoneda embedding $h\colon\EuScript{C}\hookrightarrow\operatorname{PSh}^{\mathcal{F}}(\EuScript{C})$ (\cite[\S 3.3]{Rez22}). Thus they share the same colimit:
		\[\tau_{\leqslant n-1}|\EuScript{C}|\simeq\operatorname{colim}(\EuScript{C}\xrightarrow{*}\tau_{\leqslant n-1}\EuScript{S})\simeq\operatorname{colim}(\operatorname{PSh}^{\mathcal{F}}(\EuScript{C})\xrightarrow{*}\tau_{\leqslant n-1}\EuScript{S})\simeq*(\mathbf{1})=*\]
		as $\mathbf{1}$ is a terminal object of $\operatorname{PSh}^{\mathcal{F}}(\EuScript{C})$.
		\item Since $\mathcal{F}$ is a filtering class, we have $\EuScript{C}_{/X}\in\mathcal{F}$ for every $X\in\operatorname{PSh}^{\mathcal{F}}(\EuScript{C})$. Thus $h\colon\EuScript{C}\hookrightarrow\operatorname{PSh}^{\mathcal{F}}(\EuScript{C})$ is left $n$-cofinal. The rest statement is easy.
	\end{enumerate}
	
\end{proof}
\begin{prop}
Let $\EuScript{C}$ be a small $\infty$-category, let $\kappa$ be a regular cardinal.
\begin{enumerate}[label=\myparen{\arabic*}]
	\item The $\infty$-category $(\operatorname{Ind}_{\kappa}(\EuScript{C}))^{\kappa}\subset\EuScript{P}(\EuScript{C})$ consists of objects that are retracts of representable functors \myparen{retracts of objects in the image of the Yoneda embedding $h\colon\EuScript{C}\hookrightarrow\EuScript{P}(\EuScript{C})$}.
	
	If $\EuScript{C}$ has all $\kappa$-filtered colimits, then $(\operatorname{Ind}_{\kappa}(\EuScript{C}))^{\kappa}\subset\EuScript{P}(\EuScript{C})$ consists of representables.
	
	In particular, if $\EuScript{C}$ has all filtered colimits, then $(\operatorname{Ind}\EuScript{C})^{\rm fp}\subset\EuScript{P}(\EuScript{C})$ consists of representables.
	\item The $\infty$-category $(\operatorname{sInd}(\EuScript{C}))^{\rm sfp}\subset\EuScript{P}(\EuScript{C})$ consists of retracts of representable functors.
	
	If $\EuScript{C}$ has all small sifted colimits \myparen{or just $\kappa$-filtered colimits}, then $(\operatorname{sInd}(\EuScript{C}))^{\rm sfp}\subset\EuScript{P}(\EuScript{C})$ consists of representables.
\end{enumerate}
\end{prop}
\begin{proof}
	Since objects in the image of the Yoneda embedding are completely compact, and $(\operatorname{Ind}_{\kappa}(\EuScript{C}))^{\kappa}$ is closed  under retracts in $\EuScript{P}(\EuScript{C})$, we find that retracts of representables are in $(\operatorname{Ind}_{\kappa}(\EuScript{C}))^{\kappa}$.
	
	Conversely, given $A\in(\operatorname{Ind}_{\kappa}(\EuScript{C}))^{\kappa}$. By \cite[\S 4.2]{Rez22}, the $\infty$-category $\EuScript{C}_{/A}$ is $\kappa$-filtered as $\mathcal{F}=\operatorname{Filt}_{\operatorname{Sm}_{\kappa}}$ is a filtering class (\cite[\S 10.7]{Rez22}). Since $A\simeq\operatorname{colim}_{(h_X\to A)\in\EuScript{C}_{/A}}h_X$, the relation
	\[\operatorname{id}_A\in\operatorname{Map}_{\EuScript{P}(\EuScript{C})}(A, A)\simeq\operatorname{colim}_{(h_X\to A)\in\EuScript{C}_{/A}}\operatorname{Map}_{\EuScript{P}(\EuScript{C})}(A, h_X)\]
	yields a retract $h_X\to A$, so $A$ is a retract of a representable. If $\EuScript{C}$ has all $\kappa$-filtered colimits, then by \cite[Corollary 4.4.5.16]{HTT}, $A$ is also representable.
	
	This proves (1); (2) has the same proof (using \cite[\S 10.8]{Rez22}).
\end{proof}
\begin{remark}
	Using \cite[\S 5.1 and 9.1]{Rez22}, one can prove that $(\operatorname{PSh}^{\mathcal{F}}(\EuScript{C}))^{\mathcal{F}\operatorname{-cpt}}\subset\EuScript{P}(\EuScript{C})$ consists of retracts of objects in the image of the Yoneda embedding $\EuScript{C}\hookrightarrow\EuScript{P}(\EuScript{C})$ for a class $\mathcal{F}\subset\operatorname{Cat}_{\infty}$ of small $\infty$-categories (in particular, it is essentially small); see also \cite[\S 11.1]{Rez21}. In particular, it is independent of the class $\mathcal{F}$. Cf. \cite[Lemma 2.6]{COJ}.
\end{remark}
\begin{prop}
Let $0\leqslant n\leqslant\infty$, let $\EuScript{E}$ be a cocomplete $(n, 1)$-category generated under colimits by $\EuScript{E}^{\rm sfp}$, i.e. $\EuScript{P}(\EuScript{E}^{\rm sfp})\xrightarrow{\sim}\EuScript{E}$. Then
\[
\EuScript{P}_{\Sigma}(\EuScript{E}^{\rm sfp})\simeq\operatorname{sInd}(\EuScript{E}^{\rm sfp})\xrightarrow{\sim}\EuScript{E}.\]

If $n<\infty$, we also have
\[\operatorname{Ind}(\EuScript{E}^{\rm fp})\xrightarrow{\sim}\EuScript{E}.\]
\end{prop}
\begin{proof}
	By our Main Theorem 1, each object $X$ in $\EuScript{E}$ is a geometric realisation of coproducts of objects in $\EuScript{E}^{\rm sfp}$. As any coproduct is a filtered colimit of finite coproducts, and since $\EuScript{E}^{\rm sfp}$ is closed under finite coproducts, we see that each object $X$ in $\EuScript{E}$ is in the sifted completion $\operatorname{sInd}(\EuScript{E}^{\rm sfp})$.
	
	On the other hand, as $\EuScript{E}$ is a cocomplete $(n, 1)$-category, each object $X$ in $\EuScript{E}$ is a $(\mathbf{\Delta}_{\rm s}^{\leqslant n})^{\rm op}$-shaped colimit of coproducts of objects in $\EuScript{E}^{\rm sfp}$. Since objects in $\EuScript{E}^{\rm sfp}$ are finitely presented, if $n<\infty$, we can filter this $(\mathbf{\Delta}_{\rm s}^{\leqslant n})^{\rm op}$-shaped diagram by $(\mathbf{\Delta}_{\rm s}^{\leqslant n})^{\rm op}$-shaped diagrams, each term is in $\EuScript{E}^{\rm sfp}$, so each (finite) colimit of such a diagram is in $\EuScript{E}^{\rm fp}$; the colimit of the resulting filtered diagram is $X$. By \cite[\S 9.2]{Rez22}, $\operatorname{Ind}(\EuScript{E}^{\rm fp})\xrightarrow{\sim}\EuScript{E}$.
\end{proof}
\begin{theorem}\label{yoneda-decomfun}
	Let $\mathcal{U}\subset\operatorname{Cat}_{\infty}$ be a class of small $\infty$-categories and let $\mathcal{F}:= \operatorname{Filt}_{\mathcal{U}}$. Let $\EuScript{C}$ be a small $\infty$-category admitting $\mathcal{U}$-colimits for which $\operatorname{PSh}^{\mathcal{F}}(\EuScript{C})=\operatorname{Fun}^{\rm \mathcal{U}^{\rm op}-lim}(\EuScript{C}^{\rm op}, \EuScript{S})\subset\EuScript{P}(\EuScript{C})$, and assume that $\operatorname{PSh}^{\mathcal{F}}(\EuScript{C})$ admits $\mathcal{U}$-colimits. Then for an  $\infty$-category $\EuScript{V}$ admitting $\mathcal{F}$-colimits, the restriction equivalence
	\[h^*\colon\operatorname{Fun}^{\rm \mathcal{F}-colim}(\operatorname{PSh}^{\mathcal{F}}(\EuScript{C}), \EuScript{V})\xrightarrow{\sim}\operatorname{Fun}(\EuScript{C}, \EuScript{V}), F\mapsto F\circ h\]
	restricts to an equivalence
	\[h^*\colon\operatorname{Fun}^{\rm \mathcal{F}\cup\mathcal{U}-colim}(\operatorname{PSh}^{\mathcal{F}}(\EuScript{C}), \EuScript{V})\xrightarrow{\sim}\operatorname{Fun}^{\rm \mathcal{U}-colim}(\EuScript{C}, \EuScript{V}).\]
	
	If $\operatorname{PSh}^{\mathcal{F}}(\EuScript{C})$ admits all small colimits, then $\operatorname{Fun}^{\rm \mathcal{F}\cup\mathcal{U}-colim}(\operatorname{PSh}^{\mathcal{F}}(\EuScript{C}), \EuScript{V})=\operatorname{Fun}^{\rm colim}(\operatorname{PSh}^{\mathcal{F}}(\EuScript{C}), \EuScript{V})$.
\end{theorem}
\begin{proof}
	We adapt the proof of \cite[Proposition 5.5.8.15]{HTT}. By \Cref{colim-yon} (2), we obtain a restricted functor
	\[h^*\colon\operatorname{Fun}^{\rm \mathcal{F}\cup\mathcal{U}-colim}(\operatorname{PSh}^{\mathcal{F}}(\EuScript{C}), \EuScript{V})\to\operatorname{Fun}^{\rm \mathcal{U}-colim}(\EuScript{C}, \EuScript{V}), F\mapsto F\circ h.\]
	
	So let $F\in\operatorname{Fun}^{\rm \mathcal{F}-colim}(\operatorname{PSh}^{\mathcal{F}}(\EuScript{C}), \EuScript{V})$ with $F\circ h\in\operatorname{Fun}^{\rm \mathcal{U}-colim}(\EuScript{C}, \EuScript{V})$, we want to show
	\[F\in\operatorname{Fun}^{\rm \mathcal{F}\cup\mathcal{U}-colim}(\operatorname{PSh}^{\mathcal{F}}(\EuScript{C}), \EuScript{V}).\]
	
	By \cite[Lemma 5.3.5.7]{HTT}, we can find a full embedding $\theta\colon\EuScript{V}\to\EuScript{D}$ with $\EuScript{D}$ cocomplete, such that a small diagram $v\colon I^{\triangleright}\to\EuScript{V}$ is a colimit diagram if and only if $\theta\circ v\colon I^{\triangleright}\to\EuScript{D}$ is. By \cite[Lemma 5.1.5.5]{HTT}, we can find a colimit-preserving functor $G\colon\EuScript{P}(\EuScript{C})\to\EuScript{D}$ which is a left Kan extension of $\theta\circ F\circ h$ along $\iota\circ h\colon\EuScript{C}\to\EuScript{P}(\EuScript{C})$, i.e., $G\simeq(\iota\circ h)_!(\theta\circ F\circ h)$. The situation can be depicted as the following diagram:
	\[\begin{tikzcd}[column sep=4em]
		\EuScript{C} \ar[r, "F\circ h"] \ar[d, hook, "h"']  &  \EuScript{V} \arrow[r, "\theta"]	  &	\EuScript{D}.	\\
		\operatorname{PSh}^{\mathcal{F}}(\EuScript{C})  \ar[d, hook, "\iota"']   \ar[rru, bend right=15, "G\circ\iota"]   \ar[ru, bend right=10, "F"]     &   	&	\\
		\EuScript{P}(\EuScript{C}) \ar[rruu, bend right=20, "G"']&   	&
	\end{tikzcd}\]
	
	We have $G\circ\iota\simeq\iota^*(\iota\circ h)_!(\theta\circ F\circ h)\simeq h_!h^*(\theta\circ F)$ with a canonical natural transformation $\varepsilon\colon G\circ\iota\to\theta\circ F$; we have $h^*\varepsilon\colon h^*(G\circ\iota)\xrightarrow{\sim}h^*(\theta\circ F)$. Let $J$ denote the full subcategory of $\operatorname{PSh}^{\mathcal{F}}(\EuScript{C})$ on which $\varepsilon$ restricts to an equivalence (in $\operatorname{Fun}(J, \EuScript{D})$). Clearly $J$ contains the essential image of the Yoneda functor $h\colon\EuScript{C}\hookrightarrow\operatorname{PSh}^{\mathcal{F}}(\EuScript{C})$; also since both $G\circ\iota$ and $\theta\circ F$ preserve $\mathcal{F}$-colimits, we must have $J=\operatorname{PSh}^{\mathcal{F}}(\EuScript{C})$. Thus $\varepsilon\colon G\circ\iota\xrightarrow{\sim}\theta\circ F\in\operatorname{Fun}(\operatorname{PSh}^{\mathcal{F}}(\EuScript{C}), \EuScript{D})$, and so $G\circ(\iota\circ h)\simeq\theta\circ(F\circ h)$ preserves $\mathcal{U}$-colimits.
	
	The adjoint functor theorem (\cite[Remark 5.5.2.10]{HTT}) tells that $G$ is a left adjoint, whose right adjoint factors through $\operatorname{PSh}^{\mathcal{F}}(\EuScript{C})$ by \Cref{colim-yon} (4). So $\theta\circ F\simeq G\circ\iota\colon\operatorname{PSh}^{\mathcal{F}}(\EuScript{C})\to\EuScript{D}$ is a left adjoint, which thus preserves all colimits which exist in $\operatorname{PSh}^{\mathcal{F}}(\EuScript{C})$, and so does $F$.
\end{proof}
\begin{remark}
	Another possible way to prove the above result, using \Cref{lim-decomp}, is to show: if $\operatorname{PSh}^{\mathcal{F}}(\EuScript{C})$ admits $\mathcal{U}$-colimits, then the functor $\operatorname{PSh}^{\mathcal{F}}(\EuScript{C})\to\operatorname{Cat}_{\infty}, X\mapsto\EuScript{C}_{/X}$ (which classifies the cocartesian fibration $\EuScript{C}\times_{\operatorname{PSh}^{\mathcal{F}}(\EuScript{C})}\operatorname{PSh}^{\mathcal{F}}(\EuScript{C})^{\Delta^1}\to\operatorname{PSh}^{\mathcal{F}}(\EuScript{C})$, where the left hand side is the fibre product of $h\colon\EuScript{C}\hookrightarrow\operatorname{PSh}^{\mathcal{F}}(\EuScript{C})$ and $d_1\colon\operatorname{PSh}^{\mathcal{F}}(\EuScript{C})^{\Delta^1}\to\operatorname{PSh}^{\mathcal{F}}(\EuScript{C})$, and the map is the projection followed by $d_0$) preserves $\mathcal{U}$-colimits.
\end{remark}

\section{On descent}

In this section, we briefly discuss the general notion of descent. It looks off-topic, but fits well as a natural continuation of functor-fibration correspondences and (co)limits in $\operatorname{Cat}_{\infty}$. This of course suits best in the $\infty$-topos setting.

For an $\infty$-category $\EuScript{E}$, write $\operatorname{Cart}(\EuScript{E})$ for the non-full $\infty$-subcategory of $\EuScript{E}^{\Delta^1}$ consisting of cartesian squares as morphisms. If $\EuScript{E}$ admits fibered products, then $\operatorname{Cart}(\EuScript{E})$ is the non-full $\infty$-subcategory of $\EuScript{E}^{\Delta^1}$ consisting of $d_0$-cartesian edges, where $d_0\colon\EuScript{E}^{\Delta^1}\to\EuScript{E}$ sends an arrow in $\EuScript{E}$ to its target (see \Cref{ar-cart-fib}).
\begin{corollary}
	Let $\EuScript{E}$ be an $\infty$-category that admits colimits and fibered products. Then for any functor $F\colon K\to\EuScript{E}^{\mathrm{op}}$, we have
	\[\operatorname{lim}(K\xrightarrow{F}\EuScript{E}^{\mathrm{op}}\xrightarrow{F_{\EuScript{E}}}\operatorname{CAT}_{\infty})\simeq\operatorname{Fun}_{/\EuScript{E}^{\mathrm{op}}}(K, \operatorname{Cart}(\EuScript{E})^{\mathrm{op}})\in\operatorname{CAT}_{\infty},\]
	where $F_{\EuScript{E}}\colon\EuScript{E}^{\mathrm{op}}\to\operatorname{CAT}_{\infty}, c\mapsto\EuScript{E}_{/c}$. Objects in $\operatorname{lim}(K\xrightarrow{F}\EuScript{E}^{\mathrm{op}}\xrightarrow{F_{\EuScript{E}}}\operatorname{CAT}_{\infty})$ are given by lifts of the diagram
	\[
	\begin{tikzcd}[row sep=1.5em, column sep=3.2em]
		& \operatorname{Cart}(\EuScript{E})^{\mathrm{op}} \arrow[d, hook]\\
		& (\EuScript{E}^{\Delta^1})^{\mathrm{op}} \arrow[d, "d_0^{\mathrm{op}}"]\\
		K  \ar[r, "F"]  \ar[ruu, bend left, densely dotted, "A"]    &  \EuScript{E}^{\mathrm{op}}.
	\end{tikzcd}\]
\end{corollary}
\begin{proof}
	Since the unstraightening equivalences commute with base change, we obtain a pullback square
	\[
	\begin{tikzcd}
		L \ar[r]    \ar[d]   \arrow[dr, phantom, "\pb", very near start]  & (\EuScript{E}^{\Delta^1})^{\mathrm{op}} \arrow[d, "d_0^{\mathrm{op}}"] \\
		K \ar[r, "F"]      & \EuScript{E}^{\mathrm{op}},
	\end{tikzcd}\]
	where the left vertical cocartesian fibration $L\to K$ is  classified by $K\xrightarrow{F}\EuScript{E}^{\mathrm{op}}\xrightarrow{F_{\EuScript{E}}}\operatorname{CAT}_{\infty}$. Thus
	\[\operatorname{lim}(K\xrightarrow{F}\EuScript{E}^{\mathrm{op}}\xrightarrow{F_{\EuScript{E}}}\operatorname{CAT}_{\infty})\simeq\operatorname{Fun}_{/K}^{\rm cocart}(K, L)\simeq\operatorname{Fun}_{/\EuScript{E}^{\mathrm{op}}}^{\rm cocart}(K, (\EuScript{E}^{\Delta^1})^{\mathrm{op}})\simeq\operatorname{Fun}_{/\EuScript{E}^{\mathrm{op}}}(K, \operatorname{Cart}(\EuScript{E})^{\mathrm{op}}),\]
	as desired.
\end{proof}
\begin{definition}[Descent]\label{abs-des}
	We say that an $\infty$-category $\EuScript{E}$ \emph{has descent} if $\operatorname{Cart}(\EuScript{E})$ has small colimits and the inclusion functor $\operatorname{Cart}(\EuScript{E})\hookrightarrow\EuScript{E}^{\Delta^1}$ preserves small colimits.
\end{definition}
\begin{remark}
	This formulation of having descent is due to Charles Rezk. An important result in higher topos theory says that every $\infty$-topos has descent (\cite[Theorem 6.1.3.9 and Proposition 6.1.3.10]{HTT}). Having descent is somehow complementary to the condition that colimits are universal (cf. \cite[Lemma 6.1.3.3 (5)]{HTT}).
\end{remark}
\begin{prop}\label{des-cat}
	Let $\EuScript{E}$ be an $\infty$-category that admits colimits and fibered products in which colimits are universal. Then $\EuScript{E}$ has descent if and only if the functor $F_{\EuScript{E}}\colon\EuScript{E}^{\mathrm{op}}\to\operatorname{CAT}_{\infty}, c\mapsto\EuScript{E}_{/c}$ preserves limits.
\end{prop}
\begin{proof}
	Let $F\colon K\to\EuScript{E}^{\mathrm{op}}, i\mapsto d_i$ be a functor with limit $d$ (i.e. $F$ extends to a limit diagram $\overline{F}\colon K^{\triangleleft}\to\EuScript{E}^{\mathrm{op}}$ with cone object $d$). Since colimits are universal in $\EuScript{E}$, the induced functor
	\[\EuScript{E}_{/d}\to\operatorname{lim}(K\xrightarrow{F}\EuScript{E}^{\mathrm{op}}\xrightarrow{F_{\EuScript{E}}}\operatorname{CAT}_{\infty})\simeq\operatorname{Fun}_{/\EuScript{E}^{\mathrm{op}}}(K, \operatorname{Cart}(\EuScript{E})^{\mathrm{op}})\]
	sending $c\to d$ to the diagram $K\to\operatorname{Cart}(\EuScript{E})^{\mathrm{op}}, i\mapsto(c\times_dd_i\to d_i)$ has a left-inverse given by sending a diagram $A\in\operatorname{Fun}_{/\EuScript{E}^{\mathrm{op}}}(K, \operatorname{Cart}(\EuScript{E})^{\mathrm{op}})$ to $\operatorname{colim}(K^{\mathrm{op}}\xrightarrow{A^{\rm op}}\operatorname{Cart}(\EuScript{E})\hookrightarrow\EuScript{E}^{\Delta^1})$.
	
	So $F_{\EuScript{E}}$ preserves limits if and only if it is also a right-inverse. This is the same as saying that for any diagram $A\colon K\to\operatorname{Cart}(\EuScript{E})^{\mathrm{op}}$, if we denote $\operatorname{colim}(K^{\mathrm{op}}\xrightarrow{A^{\rm op}}\operatorname{Cart}(\EuScript{E})\hookrightarrow\EuScript{E}^{\Delta^1})$ by $c\to d$, and write $A^{\rm op}(i)=:(c_i\to d_i)$, then for each $i\in K$, the square
	\[
	\begin{tikzcd}
		c_i \ar[r] \ar[d]  & c  \arrow[d]\\
		d_i  \ar[r]          & d
	\end{tikzcd}\]
	is cartesian in $\EuScript{E}$, which amounts to saying that ($\operatorname{Cart}(\EuScript{E})$ has small colimits and) $\operatorname{Cart}(\EuScript{E})\hookrightarrow\EuScript{E}^{\Delta^1}$ preserves small colimits, i.e. $\EuScript{E}$ has descent.
\end{proof}

\begin{definition}[Descent property for a section]\label{abs-des-sect}
	Let $\EuScript{C}$ be an $\infty$-category, let $\Phi\colon\EuScript{C}^{\mathrm{op}}\to\operatorname{Cat}_{\infty}$ be a functor classifying a cocartesian fibration $p\colon\EuScript{E}=\int_{\EuScript{C}^{\rm op}}\Phi\to\EuScript{C}^{\rm op}$.
	
	We say that a section $s\colon\EuScript{C}^{\rm op}\to\EuScript{E}$ of $p$ \emph{satisfies descent} if for each $c\in\EuScript{C}$ and $e\in \Phi(c)$, the functor
	\[
	\begin{aligned}
		F_{c, e}\colon(\EuScript{C}_{/c})^{\rm op}&\to\EuScript{S},\\
		(c'\xrightarrow{u}c)&\mapsto\operatorname{Map}_{\Phi(c')}(u^*e, s(c'))
	\end{aligned}\]
	preserves small limits. Here, for $u\colon c'\xrightarrow{h}c''\xrightarrow{v}c$, the map $F_{c, e}(h)$ is the composite
	\[\operatorname{Map}_{\Phi(c'')}(v^*e, s(c''))\xrightarrow{h^*}\operatorname{Map}_{\Phi(c')}(u^*e, h^*s(c''))\to\operatorname{Map}_{\Phi(c')}(u^*e, s(c')),\]
	where the latter map is induced by the canonical morphism $h^*s(c'')\to s(c')$.
	
	We say that an object $e\in \Phi(c)$ \emph{satisfies descent} if the cocartesian section $(c'\xrightarrow{u}c)\mapsto u^*e$ of the cocartesian fibration $p_{c/}\colon\EuScript{E}_{c/}\to(\EuScript{C}_{/c})^{\rm op}$ satisfies descent, i.e., for any morphism $u\colon c'\to c$ and object $e'\in \Phi(c')$, the functor
	\[
	\begin{aligned}
		F_{c', e'}\colon(\EuScript{C}_{/c'})^{\rm op}&\to\EuScript{S},\\
		(b\xrightarrow{v}c')&\mapsto\operatorname{Map}_{\Phi(b)}(v^*e', v^*u^*e)
	\end{aligned}\]
	preserves small limits.
\end{definition}
\begin{prop}
	Let $\EuScript{C}$ be an $\infty$-category, let $\Phi\colon\EuScript{C}^{\mathrm{op}}\to\operatorname{Cat}_{\infty}$ be a functor, let $c\in\EuScript{C}$ be an object. If an object $e\in \Phi(c)$ satisfies descent, then $u^*e\in\Phi(c')$ satisfies descent for every morphism $c'\xrightarrow{u}c$ in $\EuScript{C}$.
\end{prop}
\begin{prop}[Descent property for a functor]
	Let $\EuScript{C}$ be an $\infty$-category. A constant functor $\Phi\colon\EuScript{C}^{\mathrm{op}}\to\operatorname{CAT}_{\infty}$ with value $\EuScript{D}$ classifies the cocartesian fibration $p\colon\EuScript{C}^{\rm op}\times\EuScript{D}\to\EuScript{C}^{\rm op}$. Any functor $F\colon\EuScript{C}^{\rm op}\to\EuScript{D}$ can be identified with a section of the cocartesian fibration $p$. It satisfies descent if and only if $F\in\operatorname{Fun}^{\rm lim}(\EuScript{C}^{\rm op}, \EuScript{D})$.
\end{prop}
\begin{proof}
	For any $c\in\EuScript{C}, d\in\EuScript{D}$, we have
	\[
	\begin{aligned}
		F_{c, d}\colon(\EuScript{C}_{/c})^{\rm op}&\to\EuScript{S},\\
		(c'\xrightarrow{u}c)&\mapsto\operatorname{Map}_{\EuScript{D}}(d, F(c')).
	\end{aligned}\]
	Given a colimit diagram $G\colon K^{\triangleright}\to\EuScript{C}$ with cone object $c$, we have $F_{c, d}(G(i)\to c)=\operatorname{Map}_{\EuScript{D}}(d, F(G(i))), \forall i\in K$. So $F_{c, d}$ preserves small limits for all $c\in\EuScript{C}, d\in\EuScript{D}$ if and only if $F\in\operatorname{Fun}^{\rm lim}(\EuScript{C}^{\rm op}, \EuScript{D})$.
\end{proof}
\begin{remark}
	By taking $\EuScript{D}=\operatorname{Cat}_{\infty}$, we observe (in light of \Cref{des-cat}) that if an $\infty$-category $\EuScript{E}$ admits colimits and fibered products in which colimits are universal, then the section $c\mapsto(c, \EuScript{E}_{/c})$ of the cocartesian fibration $\EuScript{E}^{\rm op}\times\operatorname{Cat}_{\infty}\to\EuScript{E}^{\rm op}$ satisfies descent if and only if the $\infty$-category $\EuScript{E}$ has descent. We see that the descent property for sections includes that for $\infty$-categories in \Cref{abs-des} as a special case. See \cite{BHoNorm} for a concise description in the case $\EuScript{C}$ is the $\infty$-topos associated to a Grothendieck site.
	
	In geometric situations, \cite[Corollary 4.7.5.3]{HA} gives a nice descent criterion for $\operatorname{Cat}_{\infty}$-valued cosimplicial diagrams that are very useful in practice.
\end{remark}

\providecommand{\bysame}{\leavevmode\hbox to3em{\hrulefill}\thinspace}
\providecommand{\MR}{\relax\ifhmode\unskip\space\fi MR }
\providecommand{\MRhref}[2]{%
  \href{http://www.ams.org/mathscinet-getitem?mr=#1}{#2}
}
\providecommand{\href}[2]{#2}

\vspace{5mm}

\noindent{\textsc{School of Mathematical Sciences\\ The University of Nottingham\\ University Park\\Nottingham, NG7 2RD}}

\vspace{5mm}

\noindent{Email: pengdudp@gmail.com}

\end{document}